\documentclass[11pt,a4paper,reqno]{amsart}

\usepackage{indentfirst,enumerate,a4wide,fancyhdr,comment}

\usepackage[colorlinks=true, linktocpage=true, linkcolor=red, citecolor=red]{hyperref}

\newtheorem{Definition}{Definition}[section]
\newtheorem{Theorem}[Definition]{Theorem}

\newtheorem*{Notation}{Notation innovation}
\newtheorem{Corollary}[Definition]{Corollary}
\newtheorem{Proposition}[Definition]{Proposition}
\newtheorem{Lemma}[Definition]{Lemma}
\newtheorem{Remark}[Definition]{Remark}

\numberwithin{equation}{section}

\title[Rapid mixing for torus extensions of hyperbolic flows]{Rapid mixing and superpolynomial equidistribution for torus extensions of hyperbolic flows}
\date{\today}

\author[DAOFEI ZHANG]{Daofei Zhang}
\address{School of Mathematics and Computing Science, Guilin University of Electronic Technology, Guilin, 541004, China}
\email{Daofei.Zhang@guet.edu.cn}

\begin{document}

\begin{abstract}
In this paper, we study mixing rates for $\mathbb{T}^{d}$-extensions of hyperbolic flows. Given three closed orbits with their holonomies, we can relate them to a point in $\mathbb{R}^{d+1}$. We prove that the extension flow enjoys rapid mixing, if the associated point is an inhomogeneous Diophantine number. Under the same assumption, we also obtain the superpolynomial equidistribution, namely, a superpolynomial error term in the equidistribution of the holonomy around closed orbits. Lastly, we apply these results to a class of three-dimensional frame flows.
\end{abstract}

\maketitle

\section{Introduction}\label{sec 1}

\subsection{Statement of main results}\label{subsec 1.1}

Estimating the mixing rates of smooth flows with some hyperbolicity is a highly challenging problem. In this paper, we focus on a class of the simplest partially hyperbolic flows: torus extensions of hyperbolic flows. We begin by recalling a seminal result of Dolgopyat from \cite{Dol98b}. Let \( M \) be a compact smooth Riemannian manifold, and let \( g_t: \Lambda \to \Lambda \subset M \) be a \( C^\infty \) hyperbolic flow where $\Lambda$ is a hyperbolic basic set. Let \( \mu_\Phi \) denote the Gibbs measure associated with a Hölder continuous potential \( \Phi \) on \( \Lambda \). A Diophantine number \( \alpha \in \mathbb{R} \) means that \( |q\alpha - p| \geq C|q|^{-\gamma} \) for some constants \( C > 0 \) and \( \gamma > 0 \), for any \( p \in \mathbb{Z} \) and any nonzero \( q \in \mathbb{Z} \).

\begin{Theorem}[Dolgopyat \cite{Dol98b}]\label{Theorem 1.1}
If there exist two closed orbits $\tau_{1}, \tau_{2}$ with periods $\ell_{1}, \ell_{2}$ such that the associated number $\alpha:=\frac{\ell_{1}}{\ell_{2}}$ is a Diophantine number, then $g_{t}$ is rapidly mixing with respect to $\mu_{\Phi}$. Specifically, the quantity
$$\int E\circ g_{t}\cdot Fd\mu_{\Phi}-\int Ed\mu_{\Phi}\int Fd\mu_{\Phi}$$
decays to zero faster than any polynomial rate as $t\to\infty$ for any smooth functions $E, F$ on $M$.
\end{Theorem}

In this paper, we extend the above result to torus extensions of hyperbolic flows, focusing particularly on a class of three-dimensional frame flows. To formulate our results more precisely, we now introduce more notations. Let $g_{t}:\Lambda\to\Lambda\subset M$ and $\mu_{\Phi}$ be introduced above. Fix an integer $d\ge1$. Consider $\widehat{M}$ as a smooth principle $\mathbb{T}^{d}$-bundle over $M$ with the bundle projection $\varrho:\widehat{M}\to M$, and denote $\widehat{\Lambda}:=\varrho^{-1}(\Lambda)$. Define $f_{t}:\widehat{\Lambda}\to\widehat{\Lambda}$ as a $C^{\infty}$ $\mathbb{T}^{d}$-extension of $g_{t}$. The extension flow $f_{t}$ preserves the local product measure $\widehat{\mu}_{\Phi}$ of $\mu_{\Phi}$ and the Lebesgue measure on $\mathbb{T}^{d}$.

\begin{Notation}
Throughout this paper, we denote $\mathbb{T}^{d}=\mathbb{R}^{d}/2\pi\mathbb{Z}^{d}$. If $x,y\in\mathbb{R}^{d}$, we use $xy$ to denote their inner product, i.e., $xy:=\sum_{i=1}^{d}x_{i}y_{i}$, and we use $|x|$ to denote its $1$-norm, i.e., $|x|:=\sum_{i=1}^{d}|x_{i}|$. 
\end{Notation}

For a closed orbit $\tau$ of $g_{t}$, its holonomy class induced by $f_{t}$, being a single point set which we assume belongs to $[0,2\pi)^{d}$ and denoted by $\theta_{\tau}$. Given three closed orbits $\tau_{1}, \tau_{2}, \tau_{3}$ with periods $\ell_{1}, \ell_{2}, \ell_{3}$ and holonomies $\theta_{1},\theta_{2},\theta_{3}\in[0,2\pi)^{d}$, we can  associate them a number $(\alpha,\beta)\in\mathbb{R}^{d+1}$ with $$\alpha:=\frac{\ell_{1}-\ell_{2}}{\ell_{2}-\ell_{3}}\in\mathbb{R}\quad\text{and}\quad\beta:=\frac{1}{2\pi}(\theta_{1}-\theta_{2})-\alpha\frac{1}{2\pi}(\theta_{2}-\theta_{3})\in\mathbb{R}^{d}.$$
We recall that an  inhomogeneous Diophantine number $(\alpha,\beta)\in\mathbb{R}^{d+1}$ satisfies the inequality: $|q\alpha+m\beta-p|\ge C(|q|+|m|)^{-\gamma}$ for some constants $C>0$ and $\gamma>0$, and for all integers $p\in\mathbb{Z}$ and $(0,0)\not=(q,m)\in\mathbb{Z}^{d+1}$. The main result of this paper is as follows.

\begin{Theorem}\label{Theorem 1.2}
If there exist three closed orbits of $g_{t}$ such that the associated point $(\alpha,\beta)\in\mathbb{R}^{d+1}$ is an inhomogeneous Diophantine number, then $f_{t}$ is rapidly mixing with respect to $\widehat{\mu}_{\Phi}$, that is for any $n\in\mathbb{N}^{+}$ there exist $C>0$ and $k\in\mathbb{N}^{+}$ such that
for any $E, F\in C^{k}(\widehat{M})$ and any $t>0$, we have 
$$\bigg|\int E\circ f_t \cdot F d\widehat \mu_{\Phi}  - \int E d\widehat \mu_{\Phi} \int Fd\widehat \mu_{\Phi}\bigg|\le C||E||_{C^{k}}||F||_{C^{k}}t^{-n}.$$
\end{Theorem}

It is worth mentioning that, in contrast to \cite{Dol98b} which uses two closed orbits, our approach employs three closed orbits. The reason and advantage for this choice lie in the fact that using three closed orbits simplifies the proof and helps avoid potential complications, while still achieving the same desired effect. That said, the use of two closed orbits would not significantly broaden the applicability of the results, since assumptions of this type—whether involving two or three orbits—remain generic rather than open and dense.

Under the same assumption as in Theorem \ref{Theorem 1.2}, we also obtain the superpolynomial equidistribution, specifically, a superpolynomial error term in the equidistribution of holonomies around closed orbits. For any \(T > 0\), let \(\pi(T)\) be the collection of prime closed orbits \(\tau\) with \(\ell_{\tau} \le T\).

\begin{Theorem}\label{Theorem 1.3}
If there exist three closed orbits of $g_{t}$ such that the associated point $(\alpha,\beta)\in\mathbb{R}^{d+1}$ is an inhomogeneous Diophantine number, then for any $n\in\mathbb{N}^{+}$  there exist $C>0$ and $k\in\mathbb{N}^{+}$ such that for any $F\in C^{k}(\mathbb{T}^{d})$ and any $T>1$,
$$
\bigg|\dfrac{1}{\#\pi(T)}\sum_{\ell_\tau \leq T}F(\theta_{\tau})-\int_{\mathbb{T}^{d}}F(\theta)d\theta\bigg|\le C||F||_{C^{k}}\frac{1}{{T^{n}}}.
$$
\end{Theorem}

The assumption in Theorems \ref{Theorem 1.2} and \ref{Theorem 1.3} is quite prevalent, since the set of inhomogeneous Diophantine numbers is generic and has full Lebesgue measure \cite{Bug05}. We can apply Theorems \ref{Theorem 1.2} and \ref{Theorem 1.3} to a class of three-dimensional frame flows. For further details on this application, please refer to Subsection~\ref{subsec 1.2}.

We conclude this subsection by relating and comparing our results to prior works. 

\begin{enumerate}[$(1)$]

	\item For compact group extensions of hyperbolic diffeomorphisms, Dolgopyat demonstrated that rapid mixing can be achieved under a suitable Diophantine condition on Brin groups~\cite{Dol02}. Additionally, for torus extensions of hyperbolic diffeomorphisms, a Diophantine condition on periodic orbits can also be used to establish rapid mixing \cite{Mel18b} by Melbourne and Terhesiu.

	\item Turning to the case of flows, for skew products of hyperbolic flows,  Field-Melbourne-Nicol-Török in \cite{Fie05} proved rapid mixing for the special case of smooth \(G\)-equivariant test functions associated with a fixed unitary representation. Our Theorem \ref{Theorem 1.2} extends this result by applying to all smooth test functions. The principal distinction between these proofs lies in the fact that the Dolgopyat type estimate in \cite{Fie05} is not uniform with respect to unitary representations, whereas our approach yields a uniform estimate for the transfer operators, as established in Proposition \ref{Dolgopyat type estimate}. Moreover, we are able to derive a quantitative expression of this Dolgopyat type estimate in terms of unitary representations. Combined with Fourier analytic techniques, this enables us to extend their result from equivariant test functions to all smooth test functions.
	
	
	\item Recently, Pollicott and the author in \cite{Pol24} proved that a large class of compact group extensions of hyperbolic flows exhibit rapid mixing. Their approach is rooted in symbolic dynamics and relies on a Diophantine assumption of Brin groups. Concurrently, Cekić and Lefeuvre in \cite{Cek24} also established rapid mixing for compact group extensions of volume-preserving Anosov flows under an assumption on Brin groups. Their proof employs a distinct method, leveraging semiclassical analysis. 
	
	A key distinction between our work and \cite{Pol24} lies in the proof of the Dolgopyat type estimate (Proposition \ref{Dolgopyat type estimate}). The proof in \cite{Pol24} relies on the use of u-s paths within stable and unstable manifolds, whereas our proof is based on the use of closed orbits. This leads to substantial differences in the derivation of the Dolgopyat type estimates. Notably, the Diophantine assumptions on Brin groups (or u-s paths) in \cite{Pol24} and those on closed orbits in our paper are independent. For instance, a hyperbolic flow may possess no nontrivial u-s paths, leading to a trivial Brin group that fails to satisfy the assumption in \cite{Pol24}, while the closed orbit assumption in our paper may still be satisfied by adjusting data from finitely many points. Conversely, the Diophantine assumption in our paper may fail even when the assumption in \cite{Pol24} holds. As an example, consider an Anosov flow arising as a suspension of an Anosov diffeomorphism. If the roof function is not a coboundary, the assumption in \cite{Pol24} holds generically for extensions of the flow. However, the roof function exhibits considerable flexibility, which can cause the value $\alpha:=\frac{\ell_{1}-\ell_{2}}{\ell_{2}-\ell_{3}}$ is not a Diophantine number for any three closed orbits. As a result, the inhomogeneous Diophantine assumption never holds for any extension of such a flow. 
\end{enumerate}

\subsection{Examples}\label{subsec 1.2}

Our findings can be applied to certain frame flows on three-dimensional compact manifolds of negative curvature, as they represent $\mathbb{T}^{1}$-extensions of geodesic flows. As a statistical property, the mixing rate of frame flows not only has its own importance but can also be used to solve problems in other fields, such as the surface subgroup conjecture of Waldhausen proved in \cite{Kah12} by Kahn and Markovic. Consider $V$ as a smooth three-dimensional oriented manifold of negative curvature, and let $M=T^{1}V$ denote the unit tangent bundle over $V$. Define $g_{t}$ as the geodesic flow and $f_{t}$ as the frame flow. We particularly focus on cases where $V$ is a convex cocompact hyperbolic three-manifold derived from a Schottky group. In this scenario, the geodesic flow $g_{t}$ is an Axiom A flow, meaning it acts on its non-wandering set $\Omega$ as a hyperbolic flow, and the frame flow $f_{t}$ serves as a smooth $\mathbb{T}^{1}$-extension of $g_{t}$ \cite{Pol91}.

Let $N\ge3$ be an integer. Choose $N$ matrices $\{T_{i}\}_{i=1}^{N}\subset \text{PSL}(2,\mathbb{C})$, and denote $T_{i}=[a_{i},b_{i};c_{i},d_{i}]$, where $a_{i}d_{i}-b_{i}c_{i}=1$. Each $T_{i}$ can be regarded as a Möbius transformation $T_{i}:\widehat{\mathbb{C}}\to\widehat{\mathbb{C}}$, given by $T_{i}(z)=\frac{a_{i}z+b_{i}}{c_{i}z+d_{i}}$, where $\widehat{\mathbb{C}}:=\mathbb{C}\cup\{\infty\}$. We may assume that $c_{i}\neq0$. 

\begin{Definition}\label{Definition of Schottky group}
	The group $\Gamma$ generated by $\{T_{i}\}_{i=1}^{N}$ is called a classical Schottky group if there exist $2N$ Euclidean open discs  $\{D_{i}\}_{i=1}^{2N}$ in $\mathbb{C}$ with pairwise disjoint closures such that $T_{i}(D_{i})=\widehat{\mathbb{C}}\setminus\overline{D_{i+N}}$ for each $1\le i\le N$.
\end{Definition}

Given $\{T_{i}\}_{i=1}^{N}\subset \text{PSL}(2,\mathbb{C})$, where $T_{i}=[a_{i},b_{i};c_{i},d_{i}]$, we can associate it with a parameter point $y=(a_{1},b_{1},c_{1},d_{1},\ldots,a_{N},b_{N},c_{N},d_{N})\in\mathbb{C}^{4N}$. Let $A_{N}$ be the parameter space for which the corresponding group $\Gamma$ forms a Schottky group.

For any point $y\in A_{N}$ and its associated matrices $\{T_{i}\}_{i=1}^{N}$, it is well-known that we can interpret $\Gamma$ as a set of isometries of the hyperbolic 3-space $\mathbb{H}^{3}$. The quotient space $V=\Gamma\setminus\mathbb{H}^{3}$ constitutes a convex cocompact hyperbolic manifold. To apply Theorems \ref{Theorem 1.2} and \ref{Theorem 1.3} to the frame flow $f_{t}$, we need to demonstrate that there exist three closed geodesics such that the associated point $(\alpha,\beta)\in\mathbb{R}^{2}$ is an inhomogeneous Diophantine number. To achieve this, we define $T:\bigcup_{1\le i\le N}D_{i}\to\widehat{\mathbb{C}}$ by $T(z)=T_{i}(z)$ if $z\in D_{i}$. The following result can be found in~\cite{Gui04} or \cite{Pol91}.

\begin{Lemma}\label{Lemma 1.5}
The prime closed geodesics correspond one-to-one to the prime periodic orbits of $T$. Furthermore, assume $\{z,\cdots,T^{n-1}(z)\}$ is a prime periodic orbit of $T$, then the length $\ell>0$ and the holonomy $\theta\in[0,2\pi)$ of the corresponding prime closed geodesic are given by $(T^{n})^{\prime}(z)=e^{\ell+i\theta}$.
\end{Lemma}

\begin{Theorem}\label{Theorem 1.6}
For each $N\ge3$ and generic $y\in A_{N}$, we have the corresponding frame flow satisfies the superpolynomial equidistribution stated in Theorem \ref{Theorem 1.3} and is rapidly mixing with respect to the local product measures of the Lebesgue measure on $\mathbb{T}^{1}$ and Gibbs measures stated in Theorem \ref{Theorem 1.2}.  
\end{Theorem}
\begin{proof}
For a point $y\in A_{N}$, each $T_{i}$ can be represented as $T_{i}(z)=\frac{a_{i}z+b_{i}}{c_{i}z+d_{i}}$, where $a_{i}d_{i}-c_{i}b_{i}=1$ and $c_{i}\neq 0$. The M\"obius map $T_{i}$ has a fixed point $z_{i}=\frac{(a_{i}-d_{i})+\sqrt{(a_{i}+d_{i})^{2}-4}}{2c_{i}}$. Hence, 
$$T_{i}^{\prime}(z_{i})=\frac{1}{(c_{i}z_{i}+d_{i})^{2}}=\eta_{i}+\sqrt{\eta_{i}^{2}-1},$$
where $\eta_{i}=\frac{a_{i}+d_{i}}{2}$. We can express $T_{i}^{\prime}(z_{i})=e^{\ell_{i}+i\theta_{i}}$, with $\theta_{i}\in[0,2\pi)$. Considering three fixed points $z_{1},z_{2},z_{3}$, we require the corresponding point $(\alpha,\beta)$ to be an inhomogeneous Diophantine number, where $\alpha=\frac{\ell_{1}-\ell_{2}}{\ell_{2}-\ell_{3}}$ and $\beta=\frac{1}{2\pi}(\theta_{1}-\theta_{2})-\alpha\frac{1}{2\pi}(\theta_{2}-\theta_{3})$. The function $f:A_{N}\to \mathbb{R}^{2}$, $y\mapsto(\alpha,\beta)$, is continuous.
	
The set $B\subset\mathbb{R}^{2}$ of inhomogeneous Diophantine points is generic. By definition,
$$B=\bigcup_{C>0}\bigcup_{\gamma>0}\bigcap_{0\neq(q,m)\in\mathbb{Z}^{2}}\bigcap_{p\in\mathbb{Z}}B_{C,\gamma,q,m,p}\supset\bigcup_{\gamma>0}\bigcap_{0\neq(q,m)\in\mathbb{Z}^{2}}\bigcap_{p\in\mathbb{Z}}\bigcup_{C>0}B_{C,\gamma,q,m,p},$$	
where $B_{C,\gamma,q,m,p}=\{(\alpha,\beta)\in\mathbb{R}^{2}:|q\alpha+m\beta-p|>C(|q|+|m|)^{-\gamma}\}$. It's clear that for any $\gamma>0$, any $0\neq(q,m)\in\mathbb{Z}^{2}$ and any $p\in\mathbb{Z}$, the set $\bigcup_{C>0}B_{C,\gamma,q,m,p}$ is open and dense. Therefore, $B$ is a generic set. Now, from the definition of $f$, it's straightforward to show that the subset $f^{-1}(B)\subset A_{N}$ is also generic. Then, the result follows from Theorems \ref{Theorem 1.2}, \ref{Theorem 1.3} and Lemma~\ref{Lemma 1.5}.
\end{proof}

The aforementioned examples are frame flows of constant negative curvature. However, we can also provide numerous examples of frame flows with variable negative curvature as well. Consider a convex cocompact hyperbolic manifold $V=\Gamma\setminus\mathbb{H}^{3}$, and let $f_{t}$ be the associated frame flow. In the proof of Theorem \ref{Theorem 1.6}, we can select $\Gamma$ such that  there are three closed geodesics $\tau_{1},\tau_{2},\tau_{3}$ such that the associated point $(\alpha,\beta)\in\mathbb{R}^{2}$ is an inhomogeneous Diophantine number. Here, the curvature $\kappa$ is constant. Now, consider another curvature $\kappa^{\prime}$ that varies and satisfies $\kappa^{\prime}=\kappa$ in a neighbourhood of $\{\tau_{1},\tau_{2},\tau_{3}\}$. Let $f_{t}^{\prime}$ be the associated frame flow with respect to this new curvature $\kappa^{\prime}$. Consequently, we have $\tau_{1},\tau_{2},\tau_{3}$ are still three closed geodesics with respect to the modified curvature $\kappa^{\prime}$, and their lengths and holonomies remain unchanged. Hence, the variable negative curvature frame flow $f_{t}^{\prime}$ satisfies the superpolynomial equidistribution stated in Theorem \ref{Theorem 1.3} and is rapidly mixing with respect to the local product measures of the Lebesgue measure on $\mathbb{T}^{1}$ and Gibbs measures stated in Theorem \ref{Theorem 1.2}.  

\subsection{Outline of the paper}\label{subsec 1.3}

The proofs of Theorems \ref{Theorem 1.2} and \ref{Theorem 1.3} follow the same strategy as in \cite{Pol24}, which necessitates a non-trivial extension of a classical approach involving a symbolic model for \(f_t\) and the use of transfer operators. Below we outline the construction of this paper and highlight the novelty of our proof. 

\begin{itemize}
\item In \S \ref{sec 2}, we firstly collect basic definitions of hyperbolic flows and their group extensions. We then define the symbolic flow we want to study. In the end of this section, we leverage a Markov partition of the underlying hyperbolic flow to related the extension flow to a symbolic flow that we defined in the last subsection.

\item In \S \ref{sec 3}, we further reduce the results to a Dolgopyat type estimate (Proposition \ref{Dolgopyat type estimate}) for the associated transfer operator \(\mathcal{L}_{s,m}\), where the variable \(s \in \mathbb{C}\) comes from the Laplace transform of the correlation function and the variable \(m \in \mathbb{Z}^{d}\) corresponds to the irreducible representations of \(\mathbb{T}^{d}\). When \(s=0\), the operator \(\mathcal{L}_{0,m}\) corresponds to the case of \(\mathbb{T}^{d}\) extensions of hyperbolic diffeomorphisms considered in \cite{Mel18b}. When \(m=0\), the operator \(\mathcal{L}_{s,0}\) corresponds to the classical case of hyperbolic flows considered in~\cite{Dol98b}. In other words, this operator \(\mathcal{L}_{s,m}\) combines characteristics of these two types of systems. Different from the estimate in \cite{Fie05}, where the spectral radius of \(\mathcal{L}_{s,m}\) could only be quantified by the variable \(s\), we show in Proposition \ref{Dolgopyat type estimate} that the spectral radius of \(\mathcal{L}_{s,m}\) can be quantified by both the variables \(s\) and \(m\). This improvement is key to proving that rapid mixing of \(f_t\) holds for all smooth functions. 

\item In \S \ref{sec 4}, we prove Proposition \ref{Dolgopyat type estimate}.  A key idea that makes the proof works is the particular choice of norm used on each of the components associated to the irreducible representations of $\mathbb{T}^{d}$. The main ingredient in the proof is a cancellation estimate stated in Lemma \ref{Lemma 3.4.4}, whose proof will be given in \S \ref{sec 5}.
\end{itemize}

\begin{Remark}
The argument can potentially be adapted to more general non-uniformly hyperbolic systems, such as those studied in \cite{Mel18}. However, since our primary application concerns frame flows, which are torus group extensions of uniformly hyperbolic systems, we focus exclusively on uniformly hyperbolic flows. This restriction is sufficient for our purposes.
\end{Remark}

\noindent\textit{Acknowledgments}. I would like to express my sincere gratitude to Professor Mark Pollicott and Professor Ian Melbourne for their insightful suggestions and valuable comments. I am also deeply thankful to the referees for their constructive feedback and recommendations. Furthermore, I gratefully acknowledge financial support from the ERC Advanced Grant 833802-Resonances.

\section{Basic definitions, Markov partitions and symbolic dynamics}\label{sec 2}

\subsection{Definitions of hyperbolic flows and their group extensions}\label{subsec 2.1}

Let $M$ be a smooth compact Riemannian manifold, and let $TM$ be its tangent bundle. Let $g_{t}:M\to M$ be a $C^{\infty}$ flow, and let $\Lambda\subset M$ be a compact $g_{t}$-invariant set. We assume $g_{t}:\Lambda\to\Lambda$ is a hyperbolic flow, meaning that:
\begin{itemize}
	\item there is a $Dg_{t}$-invariant splitting $T_{\Lambda}M=E^{s}\oplus E^{c}\oplus E^{u}$ over $\Lambda$, and exist constant $C>0$ and $\delta>0$ such that $E^{c}$ is the line bundle tangent to $g_{t}$, and $||Dg_{t}|_{E^{s}}||,\  ||Dg_{-t}|_{E^{u}}||\le Ce^{-\delta t}$ for all $t\ge0$;
	\item $\Lambda$ consists of more than a single closed orbit, and closed orbits of $g_{t}$ in $\Lambda$ are dense;
	\item $g_{t}$ is transitive on $\Lambda$ which means for any open sets $U,V$ there exists arbitrarily large $t>0$ such that $U\cap g_{t}V\not=\emptyset$, and there is an open set $\mathcal{U}\supset\Lambda$ such that $\Lambda=\bigcap_{t=-\infty}^{\infty}g_{t}\mathcal{U}$.
\end{itemize}
 	
For a continuous real-valued function \(\Phi\) on \(\Lambda\), its pressure \(P(\Phi)\) is defined as the supremum of \(h_{\nu}(g_{1}) + \int_{\Lambda} \Phi \, d\nu\) taken over all \(g_{t}\)-invariant probability measures \(\nu\), where \(h_{\nu}(g_{1})\) is the measure-theoretic entropy of \(\nu\) with respect to the time-one map \(g_{1}\) \cite{Wal82}. Let \(\Phi\) be a H\"older continuous real-valued function on \(\Lambda\), it is well known \cite{Bow75} that there exists a unique \(g_{t}\)-invariant probability measure \(\mu_{\Phi}\) on \(\Lambda\) that satisfies \(h_{\mu_{\Phi}}(g_{1}) + \int_{\Lambda} \Phi \, d\mu_{\Phi} = P(\Phi)\). We call \(\mu_{\Phi}\) the Gibbs measure of \(\Phi\).

After introducing the necessary concepts about hyperbolic flows, we now shift our focus to the primary research subject of this paper: $\mathbb{T}^{d}$-extensions of $g_t$. These flows are no longer hyperbolic and represent one of the simplest examples of partially hyperbolic flows. Let $\widehat{M}$ be a smooth  principle $\mathbb{T}^{d}$-bundle over $M$ with the bundle projection $\varrho:\widehat{M}\to M$. We consider a $C^{\infty}$ flow $f_{t}$ on $\widehat{M}$ and denote $\widehat{\Lambda}:=\varrho^{-1}(\Lambda)$.

\begin{Definition}
We say $f_{t}:\widehat{\Lambda}\to\widehat{\Lambda}$ is a $\mathbb{T}^{d}$-extension of $g_{t}$ if 
\begin{enumerate}
\item for any $t\in\mathbb{R}$, we have $\varrho\circ f_{t}=g_{t}\circ\varrho$ on $\widehat{\Lambda}$;

\item for any $\theta\in \mathbb{T}^{d}$ and any $t\in\mathbb{R}$, we have $f_{t}\circ R_{\theta}=R_{\theta}\circ f_{t}$ on $\widehat{M}$, where $R_{\theta}$ presents the $\theta$-rotation.
\end{enumerate}
\end{Definition}

The Gibbs measure $\mu_{\Phi}$ of $\Phi$ on $\Lambda$ naturally induces a probability measure $\widehat{\mu}_{\Phi}$ on $\widehat{\Lambda}$ which is the local product of $\mu_{\Phi}$ and the Lebesgue measure on $\mathbb{T}^{d}$. By definition, we have $\widehat{\mu}_{\Phi}=\int_{\Lambda}\text{Leb}_{x}d\mu_{\Phi}(x)$, where $\text{Leb}_{x}$ is the push-forward of the Lebesgue measure on $\mathbb{T}^{d}$ to the fibre $\varrho^{-1}(x)$. Then, it is not difficult to show that the extension flow $f_{t}$ preserves $\widehat{\mu}_{\Phi}$. We are interested in the asymptotic behaviour of the following quantity.

\begin{Definition}
	For $E,F : \widehat \Lambda \to \mathbb C$, their correlation function with respect to $\widehat{\mu}_{\Phi}$ is defined by 
	$$
	\rho_{E,F}(t): = \int_{\widehat{\Lambda}} E\circ f_t  \cdot F d\widehat \mu_{\Phi}  - \int_{\widehat{\Lambda}}  E d\widehat \mu_{\Phi} \int_{\widehat{\Lambda}}  Fd\widehat \mu_{\Phi}, \hbox{ for } t \in \mathbb R.
	$$
\end{Definition}

The flow $f_t : \widehat \Lambda \to \widehat \Lambda$ is  said to be mixing with respect to  $\widehat \mu_{\Phi}$ if  $\rho_{E,F}(t) \to 0$ as $t \to +\infty$ for any $E$ and $F\in L^{2}(\widehat{\mu}_{\Phi})$. In this paper, we will show that a certain Diophantine condition on closed orbits leads to a stronger effective estimate for $\rho_{E,F}(t)$ of the following form.

\begin{Definition}
	We say $f_{t}$ is rapidly mixing with respect to $\widehat \mu_{\Phi}$ if 
	for any $n\in\mathbb{N}^{+}$ there exist $C>0$ and $k\in\mathbb{N}^{+}$ such that
	for any $E, F\in C^{k}(\widehat{M})$ and any $t>0$, we have $|\rho_{E,F}(t)| \leq C||E||_{C^{k}}||F||_{C^{k}}t^{-n}$.
\end{Definition}

Given a closed orbit $\tau$ of $g_{t}$ with period $\ell_{\tau}$, for any $x\in\tau$ and any trivialization at $x$ there exists a unique $\theta_{\tau}(x)\in\mathbb{T}^{d}$ such that
$$f_{\ell_{\tau}}(x,0)=(g_{\ell_{\tau}}(x),\theta_{\tau}(x))=(x,\theta_{\tau}(x)).$$
Since \(f_{t}\) is a \(\mathbb{T}^{d}\) extension, it is not difficult to show that the quantity \(\theta_{\tau}(x)\) is independent of the choices of \(x \in \tau\) and the trivialization at \(x\). This is because if we take another trivialization, i.e., a coordinate $\theta^{\prime}=\theta+\theta_{0}$, then $(x,0)$ in the changed coordinate system $(x,\theta^{\prime})$ corresponds $(x,-\theta_{0})$ in the old coordinate system $(x,\theta)$. By the definition of a \(\mathbb{T}^{d}\) extension, in the changed coordinate system we then have $f_{\ell_{\tau}}(x,0)=(x,-\theta_{0}+\theta_{\tau}(x)+\theta_{0})=(x,\theta_{\tau}(x))$.  Thus, each closed orbit \(\tau\) of \(g_{t}\) uniquely determines an element \(\theta_{\tau} = \theta(x) \in \mathbb{T}^{d}\), which we call the holonomy of \(\tau\). Given three closed orbits \(\tau_{1}, \tau_{2}, \tau_{3}\) of \(g_{t}\) with periods \(\ell_{1}, \ell_{2}, \ell_{3}\) and holonomies \(\theta_{1}, \theta_{2}, \theta_{3}\), by taking a lift, we can think of \(\theta_{1}, \theta_{2}, \theta_{3}\) as elements in \([0, 2\pi)^{d}\). We then associate them with a number \((\alpha, \beta) \in \mathbb{R}^{d+1}\) by 
$$\alpha:=\frac{\ell_{1}-\ell_{2}}{\ell_{2}-\ell_{3}}\in\mathbb{R}\quad\text{and}\quad\beta:=\frac{1}{2\pi}(\theta_{1}-\theta_{2})-\alpha\frac{1}{2\pi}(\theta_{2}-\theta_{3})\in\mathbb{R}^{d}.$$
The following concept is mentioned in Theorem \ref{Theorem 1.2} and Theorem \ref{Theorem 1.3}.

\begin{Definition}\label{Definition 2.4}
We say $(\alpha, \beta)\in \mathbb{R}^{d+1}$ is an inhomogeneous Diophantine number if for some $C>0$ and $\gamma>0$ we have $|q\alpha+m\beta-p|\ge C(|q|+|m|)^{-\gamma}$ for any $p\in\mathbb{Z}$ and any $0\not=(q,m)\in\mathbb{Z}^{d+1}$.
\end{Definition}

The concept of inhomogeneous Diophantine number is introduced in the context of Diophantine approximation. It is known that the set of inhomogeneous Diophantine numbers in $\mathbb{R}^{d+1}$ is generic and has full Lebesgue measure \cite{Bug05}  by Bugeaud and Laurent.

\subsection{Suspension flows of $\mathbb{T}^{d}$ skew products over two-sided subshifts}\label{subsec 2.2}

Using a Markov partition for the underlying hyperbolic flow \(g_{t}\), the extension flow \(f_{t}\) can be represented as the suspension flow of a \(\mathbb{T}^{d}\)-skew product over a two-sided subshift of finite type. This connection allows us to address the rapid mixing of \(f_{t}\) by linking it to an analogous statement about the rapid mixing of the symbolic flow. In this subsection, we define the symbolic flow and outline its fundamental properties.

We first review classical subshifts of finite type and their associated symbolic models, with~\cite{Par90} serving as a helpful reference. Let $A$ be an $N\times N$ matrix of zeros and ones. We assume $A$ is aperiodic, meaning that some power of $A$ is a positive matrix. Let $X := X_{A}$ be the two-sided symbolic space associated with $A$, defined as $X = \{x=(x_{i})_{i=-\infty}^{\infty}\in\{1,\cdots,N\}^{\mathbb{Z}}: A(x_{i}, x_{i+1}) = 1\}$, and let $\sigma:X\to X$ be the two-sided subshift, defined as $\sigma(x)_{i} = x_{i+1}$. Given $\lambda \in (0, 1)$, we can define a metric $d_{\lambda}$ on $X$ by $d_{\lambda}(x, y) := \lambda^{N(x,y)}$, where $N(x,y) = \min\{|i| : x_{i} \neq y_{i}\}$. We say $\lambda\in(0,1)$ is a metric constant on $X$ means that $X$ is assigned the metric $d_{\lambda}$. It is not difficult to show that $\sigma:X\to X$ is topologically mixing, which means for any two open sets $U,V\subset X$ and any large enough $n\in\mathbb{N}^{+}$ we have $U\cap\sigma^{n}V\not=\emptyset$, if and only if $A$ is aperiodic. 

\begin{Definition}
	Let $\lambda\in(0,1)$ be a metric constant on $X$.
	\begin{itemize}
		\item Denote by $F_{\lambda}(X)$ the Banach space of all Lipschitz continuous complex-valued functions on $(X,d_{\lambda})$ with respect to the Lipschitz norm $||\cdot||_{\text{Lip}}:=|\cdot|_{\text{Lip}} + |\cdot|_{\infty}$, where $|\cdot|_{\text{Lip}}$ is the Lipschitz semi-norm and $|\cdot|_{\infty}$ is the supremum norm. 
		
		\item Denote by $F_{\lambda}(X, \mathbb{R})$ the set of real-valued functions in $F_{\lambda}(X)$. For any $n\in\mathbb{N}^{+}$, denote by $F_{\lambda}(X, \mathbb{R}^{n})$ the set of functions $h=(h_{i})_{i=1}^{n}:X\to\mathbb{R}^{n}$ with $h_{i}\in F_{\lambda}(X, \mathbb{R})$.
	\end{itemize}
\end{Definition}

For each $n\in\mathbb{N}^{+}$, the $n$-cylinders are sets of the form $[x_{-n+1}\cdots x_{n-1}]_{n} = \{y \in X : y_{i} = x_{i}, |i| \le n-1\}$. For each $n\in\mathbb{N}^{+}$, let $F_{n}(X)$ be the set of functions in $F_{\lambda}(X)$ that are constant on $n$-cylinders. The functions in \(F_{n}(X)\) are referred to as locally constant functions. All of the above objects can be similarly defined for a one-sided subshift of finite type \(\sigma:X^{+}\to X^{+}\), where \(X^{+}=\{x=(x_{i})_{i\ge0}\in\{1,\ldots,N\}^{\mathbb{N}} : A(x_{i}, x_{i+1})=1\}\), using the same notations but replacing \(X\) with \(X^{+}\).

For a potential \(\varphi\in F_{\lambda}(X,\mathbb{R})\), it is well-known \cite{Bow08} that there exists a unique equilibrium state of \(\varphi\) on \(X\), which is called the Gibbs measure of \(\varphi\) on \(X\), denoted as \(\mu_{\varphi}\). By adding  a coboundary \(h\circ\sigma-h\) to $\varphi$ if necessary, we can assume \(\varphi\) depends only on future coordinates, and thus \(\varphi\in F_{\lambda}(X^{+},\mathbb{R})\).\footnote{ There is a cost on the regularity to assume  $\varphi$ only depends on future coordinates. Actually, by adding  a coboundary \(h\circ\sigma-h\) to $\varphi$, the new potential $\varphi+h\circ\sigma-h$ belongs to $F_{\sqrt{\lambda}}(X,\mathbb{R})$. However, since $F_{\lambda}(X,\mathbb{R})\subset F_{\sqrt{\lambda}}(X,\mathbb{R})$, after replacing $\lambda$ by $\sqrt{\lambda}$ which will not affect the subsequent proof, we have \(\varphi\in F_{\lambda}(X,\mathbb{R})\).} Then, there exists a unique equilibrium state of \(\varphi\) on \(X^{+}\), which is called the Gibbs measure of \(\varphi\) on \(X^{+}\), denoted as \(\mu\). Let \(\Pi_{+}:X\to X^{+}\) be the coordinate projection. It is known that \(\Pi_{+}^{*}\mu_{\varphi}=\mu\) \cite{Bow08}. For each \(n\in\mathbb{N}^{+}\), let \(\varphi_{n}=\sum_{i=0}^{n-1}\varphi\circ\sigma^{i}\). The Gibbs measure \(\mu\) satisfies the following Gibbs property, see also \cite{Par90}.

\begin{Lemma}\label{Gibbs property}
There exists $C_{1}\geq 1$ such that 
$$C_{1}^{-1}\le \dfrac{\mu[x_{0}\cdots x_{n-1}]_{n}}{e^{\varphi_{n}(x)-nP(\varphi)}}\le C_{1},$$
for any $n\in\mathbb{N}^{+}$ and any $x\in X^{+}$ where $P(\varphi)$ presents the pressure of $\varphi$.
\end{Lemma}

Let $\sigma:X\to X$ be a two-sided subshift. Given a function $\Theta\in F_{\lambda}(X, \mathbb{R}^{d})$, let $\widehat{\sigma}:\widehat{X}\to\widehat{X}$ be the skew product, namely, 
$$
\widehat{\sigma}:\widehat{X}\to\widehat{X},\quad \widehat{\sigma}(x,\theta)=(\sigma(x),\Theta(x)+\theta\mod 2\pi),
$$
where $\widehat{X}=X\times \mathbb{T}^{d}$. Let $d\theta$ be the Lebesgue measure on $\mathbb{T}^{d}$. It is easy to see that $\widehat{\mu}_{\varphi}:={\mu}_{\varphi}\times d\theta$ is $\widehat{\sigma}$-invariant. Given $r\in F_{\lambda}(X,\mathbb{R}^{+})$, we can think of $r$ is a continuous function on $\widehat{X}$ by setting $r(x,\theta)=r(x)$. Then, we define the suspension space as $\widehat{X}_{r}:=\{(x,\theta,u):0\le u\le r(x)\}/\sim$, where $(x,\theta,r(x))\sim(\sigma(x),\Theta(x)+\theta,0)$. Let $\phi_{t}:\widehat{X}_{r}\to\widehat{X}_{r}$
be the suspension flow of $\widehat{\sigma}$ under $r$, namely,
$$
\phi_{t}:\widehat{X}_{r}\to \widehat{X}_{r},\quad \phi_{t}(x,\theta,u)=(x,\theta, u+t),\quad t\ge0,
$$
with respect to $\sim$ on $\widehat{X}_{r}$. The $\widehat{\sigma}$-invariant probability measure $\widehat{\mu}_{\varphi}$ on $\widehat{X}$ induces a natural probability measure $\widehat{\mu}_{\varphi}\times$Leb on $\widehat{X}_{r}$ by
$$
\int_{\widehat{X}_{r}}Fd(\widehat{\mu}_{\varphi}\times\text{Leb})=\dfrac{1}{\int_{\widehat{X}}rd\widehat{\mu}_{\varphi}}\int_{\widehat{X}}\int_{0}^{r(x)}F(x,\theta,u)dud\widehat{\mu}_{\varphi}(x,\theta),\quad F\in C(\widehat{X}_{r}).
$$ 
It can be shown that $\widehat{\mu}_{\varphi}\times\text{Leb}$ is $\phi_{t}$-invariant \cite{Par90}.

With a little abuse of notations, we still use $\rho_{E,F}$ to denote  the correlation function of test functions $E$ and $F\in C(\widehat{X}_{r})$ for the symbolic flow $\phi_{t}$. In other words, 
$$\rho_{E,F}(t)=\int_{\widehat{X}_{r}}E\circ\phi_{t}\cdot Fd(\widehat{\mu}_{\varphi}\times\text{Leb})-\int_{\widehat{X}_{r}}Ed(\widehat{\mu}_{\varphi}\times\text{Leb})\int_{\widehat{X}_{r}}Fd(\widehat{\mu}_{\varphi}\times\text{Leb}).$$
We can also study the rate of mixing of $\phi_{t}$ with respect to $\widehat{\mu}_{\varphi}\times$Leb. However, since the symbolic space $X$ lacks differential structure, the test functions $E$ and $F$ we should consider are a little different from the case of $f_{t}$.

\begin{Definition}\label{Definition 2.7}
	We will use the following notation.
	\begin{itemize}
		\item For a complex-valued function $E$ on $\widehat{X}_{r}$, define 
		$$|E|_{\lambda}:=\sup_{\theta\in \mathbb{T}^{d}}\sup_{x\not=y}\sup_{ u\in[0,\min\{r(x),r(y)\}]}\dfrac{|E(x,\theta,u)-E(y,\theta,u)|}{d_{\lambda}(x,y)}.$$
		Define $||E||_{\lambda}:=|E|_{\infty}+|E|_{\lambda}$ where $|E|_{\infty}:=\sup_{(x,\theta,u)}|E(x,\theta,u)|$.
		
		\item For each $k\in\mathbb{N}^{+}$, let $F^{k}_{\lambda}(\widehat{X}_{r})$ be the set of all complex-valued functions $E$ on $\widehat{X}_{r}$ which are $C^{k}$ with respect to $(\theta,u)$ and $||E||_{\lambda,k}<\infty$ where $||E||_{\lambda,k}:=\sup_{k_{1}+k_{2}\le k}||\frac{\partial^{k_{1}+k_{2}}E}{\partial \theta^{k_{1}}\partial u^{k_{2}}}||_{\lambda}$.
	\end{itemize}
\end{Definition}

This leads to the following definition of rapid mixing which is appropriate at the symbolic level.

\begin{Definition}\label{Definition 2.8}
We say $\phi_{t}$ is rapidly mixing with respect to $\widehat{\mu}_{\varphi}\times$Leb if for any $n\in\mathbb{N}^{+}$ there exist $C>0$ and $k\in\mathbb{N}^{+}$ such that for any $E,F\in F_{\lambda}^{k}(\widehat{X}_{r})$ and any $t>0$ we have $|\rho_{E,F}(t)|\le C||E||_{\lambda,k}||F||_{\lambda,k}t^{-n}$.
\end{Definition}

\subsection{Relating the extension flow to a symbolic flow}\label{subsec 2.3}

The foundation for this subsection has been laid out in \cite[\S 6.5]{Pol24}. However, we find it necessary to briefly revisit these details. This serves not only to ensure the completeness of our discussion but also to introduce the essential notations required to translate the Diophantine condition on closed orbits for the extension flow \(f_{t}\) into a corresponding Diophantine condition for the symbolic flow \(\phi_{t}\).

Let $f_{t}$ be a $C^{\infty}$ $\mathbb{T}^{d}$-extension of a hyperbolic flow $g_{t}$. We assume that $f_{t}$ satisfies the assumption of Theorem \ref{Theorem 1.2}, that is there exist three closed orbits such that the associated point $(\alpha,\beta)$ is an inhomogeneous Diophantine number. By a classical result of Bowen \cite{Bow73}, we can construct a Markov section \(\{R_{i}\}_{i=1}^{N}\) for the flow \( g_{t}:\Lambda \to \Lambda \) with arbitrarily small size. While the precise definition of a Markov section is not necessary for our purposes, it suffices to note that it is a collection of finitely many special local cross-sections of the hyperbolic flow \( g_{t} \) such that the first return map exhibits suitable hyperbolic properties.  More concretely, let \( R := \bigcup_{i} R_{i} \), and let \( P: R \to R \) denote the first return map, with \( r: R \to \mathbb{R}^{+} \) representing the first return time. Through a natural semi-conjugacy \( \Pi: X \to R \), we can identify \( P: R \to R \) with a two-sided subshift \( \sigma: X \to X \). Likewise, \( r \) can be regarded as a function on \( X \) by composition with \( \Pi \), i.e., \( r \circ \Pi \).

This Markov section $\{R_{i}\}_{i=1}^{N}$ for the underlying hyperbolic flow $g_{t}$ naturally gives rise to a finite number of local cross-sections $\{\varrho^{-1}(R_{i})\}_{i=1}^{N}$ for the extension flow $f_{t}:\widehat{\Lambda}\to\widehat{\Lambda}$, where $\varrho:\widehat{M}\to M$ serves as the bundle projection. For each $\varrho^{-1}(R_{i})$, we can trivialize it as $R_{i}\times\mathbb{T}^{d}$. Let $\widehat{R}:=R\times \mathbb{T}^{d}$. By construction, the first return time function with respect to $f_{t}$ on $\widehat{R}$ is equal to $r$. Let $\widehat{P}:\widehat{R}\to \widehat{R}$ be the first return map. In \cite[Equation (6.5.2)]{Pol24}, we have seen that $\widehat{P}$ is a skew product of $P$ with a skew function $\Theta:R\to \mathbb{R}^{d}$. Similarly, we can think of $\Theta$ as a function on $X$ by considering $\Theta\circ\Pi$. Since $P:R\to R$ is a two-sided subshift $\sigma:X\to X$ via the semi-conjugacy $\Pi$, we can regard $\widehat{P}$ as the skew product $\widehat{\sigma}:\widehat{X}\to\widehat{X}$ with the skew function $\Theta$. 

Let $\phi_{t}:\widehat{X}_{r}\to\widehat{X}_{r}$ be the suspension flow of $\widehat{\sigma}:\widehat{X}\to\widehat{X}$ under $r$. Now, it naturally identities the extension flow $f_{t}$ with the symbolic flow $\phi_{t}$. As explained in \cite[\S 6.5]{Pol24}, we can assume $r$ and $\Theta$ are functions on the corresponding one-sided space $X^{+}$, and $r\in F_{\lambda}(X^{+},\mathbb{R}^{+})$ as well as $\Theta\in F_{\lambda}(X^{+},\mathbb{R}^{d})$ for a suitable choice of the metric constant $\lambda\in(0,1)$. The potential $\Phi$ on $\Lambda$ induces a potential $\varphi$ on $X$. Let $\mu_{\varphi}$ be the Gibbs measure of $\varphi$, and let $\widehat{\mu}_{\varphi}\times$Leb be the associated $\phi_{t}$-invariant measure on $\widehat{X}_{r}$. The following result is from \cite[Lemma 6.5.4]{Pol24}.

\begin{Lemma}\label{Lemma 2.9}
If $\phi_{t}$ is rapidly mixing with respect to $\widehat{\mu}_{\varphi}\times$Leb, then $f_{t}$ is rapidly mixing with respect to $\widehat{\mu}_{\Phi}$. 
\end{Lemma}

For any $n\in\mathbb{N}^{+}$, denote by $r_{n}:=\sum_{i=0}^{n-1}r\circ\sigma^{n}$ and $\Theta_{n}:=\sum_{i=0}^{n-1}\Theta\circ\sigma^{n}$. The Diophantine assumption in Theorem \ref{Theorem 1.2} for the extension flow $f_{t}$ implies the following condition on $\phi_{t}$.

\begin{Lemma}\label{Lemma 2.10}
After modifying the symbolic coding, there exist three periodic points $\{x_{1},x_{2},x_{3}\}$ of $\sigma$ with the same period $N^{\prime}\in\mathbb{N}^{+}$ such that the associated point $(\alpha^{\prime},\beta^{\prime})\in\mathbb{R}^{d+1}$ is an inhomogeneous Diophantine number, where
	$$\alpha^{\prime}=\dfrac{r_{N^{\prime}}(x_{1})-r_{N^{\prime}}(x_{2})}{r_{N^{\prime}}(x_{2})-r_{N^{\prime}}(x_{3})}\quad\text{and}\quad\beta^{\prime}=\dfrac{1}{2\pi}(\Theta_{N^{\prime}}(x_{1})-\Theta_{N^{\prime}}(x_{2}))-\dfrac{1}{2\pi}\alpha^{\prime}(\Theta_{N^{\prime}}(x_{2})-\Theta_{N^{\prime}}(x_{3})).$$
\end{Lemma}
\begin{proof}
	Let $\tau_{1},\tau_{2},\tau_{3}$ be three closed orbits of $g_{t}$, with periods $\ell_{1},\ell_{2},\ell_{3}$ and holonomies $\theta_{1},\theta_{2},\theta_{3}$. Let $(\alpha,\beta)\in\mathbb{R}^{d+1}$ be the corresponding point, which means
	$$\alpha=\dfrac{\ell_{1}-\ell_{2}}{\ell_{2}-\ell_{3}}\in\mathbb{R}\quad\text{and}\quad\beta=\dfrac{1}{2\pi}(\theta_{1}-\theta_{2})-\dfrac{1}{2\pi}\alpha(\theta_{2}-\theta_{3})\in\mathbb{R}^{d}.$$
	For each closed orbit $\tau_{j}$, there exists a periodic point $x_{j}\in X$ with $\sigma^{N_{j}}(x_{j})=x_{j}$ such that
	$$r_{N_{j}}(x_{j})=\ell_{j}\quad\text{and}\quad \Theta_{N_{j}}(x_{j})=\theta_{j}\text{ mod }2\pi.$$
	Consequently, $\alpha^{\prime}=\alpha$ and $\beta^{\prime}=\beta\text{ mod }1$ where
	$$\alpha^{\prime}=\dfrac{r_{N_{1}}(x_{1})-r_{N_{2}}(x_{2})}{r_{N_{2}}(x_{2})-r_{N_{3}}(x_{3})}\quad\text{and}\quad\beta^{\prime}=\dfrac{1}{2\pi}(\Theta_{N_{1}}(x_{1})-\Theta_{N_{2}}(x_{2}))-\dfrac{1}{2\pi}\alpha^{\prime}(\Theta_{N_{2}}(x_{2})-\Theta_{N_{3}}(x_{3})).$$
	In particular, by definition, $(\alpha,\beta)$ being an inhomogeneous Diophantine number implies $(\alpha^{\prime},\beta^{\prime})$ is also an inhomogeneous Diophantine number.
	
It remains to show that these periodic points have the same period, namely \(N_{1} = N_{2} = N_{3}\). This can be done as follows. For a fixed closed orbit of \(g_{t}\), say \(\tau\), associated with a periodic point \(x \in X\) of \(\sigma\) with period \(n\), we have that \(\tau\) must pass through sections \(R_{x_{0}}\) and \(R_{x_{1}}\). We can then insert an arbitrary finite number of local cross-sections between \(R_{x_{0}}\) and \(R_{x_{1}}\), which are parallel to \(R_{x_{0}}\). It is easy to see that the new section is still a Markov section, and the closed orbit \(\tau\) associated with a periodic point \(x \in X\) then has period \(n+q\), where \(q\) is the number of parallel local cross-sections we insert. Consequently, we can always ensure \(N_{1} = N_{2} = N_{3}\) by inserting the appropriate number of parallel sections as described above.
\end{proof}

\section{A Dolgopyat type estimate}\label{sec 3}

To prove Theorem \ref{Theorem 1.2}, by Lemma \ref{Lemma 2.9}, we need to show that the symbolic flow \(\phi_{t}\) is rapidly mixing with respect to \(\widehat{\mu}_{\varphi} \times \text{Leb}\). The proof follows the same approach as that in \cite{Pol24}. The standard strategy for proving rapid mixing of \(\phi_{t}\) involves investigating the Laplace transform of its correlation function. By demonstrating that the Laplace transform can be analytically extended to a specific region in the left half-plane, one can derive a decay rate for the correlation function, which depends on the shape of this region. The Laplace transform can be expressed as a sum of iterations of transfer operators, and the analytical extension can be established through a Dolgopyat type estimate of these transfer operators, as stated in Proposition \ref{Dolgopyat type estimate}. Furthermore, following the same line as the proof of \cite[Theorem 4]{Pol24} in \cite[\S 8]{Pol24}, the estimate established in Proposition \ref{Dolgopyat type estimate} also allows us to ensure the superpolynomial equidistribution of \(\phi_{t}\) and thus the superpolynomial equidistribution of \(f_{t}\) as well. This, in turn, proves Theorem \ref{Theorem 1.3}. This section is dedicated to fulfilling this objective.

Recalling the correlation function for $E, F\in F_{\lambda}^{k}(\widehat{X}_{r})$ as
$$
\rho_{E,F}(t)=\int_{\widehat{X}_{r}} E\circ\phi_{t}\cdot Fd\widehat{\mu}_{\varphi}\times\text{Leb}-\int_{\widehat{X}_{r}}Ed\widehat{\mu}_{\varphi}\times\text{Leb}\int_{\widehat{X}_{r}} Fd\widehat{\mu}_{\varphi}\times\text{Leb}.
$$
As usual, by replacing $E$ with $E-\int_{\widehat{X}_{r}}Ed\widehat{\mu}_{\varphi}\times\text{Leb}$, we can assume $\int_{\widehat{X}_{r}}Ed\widehat{\mu}_{\varphi}\times\text{Leb}=\int_{\widehat{X}_{r}}Fd\widehat{\mu}_{\varphi}\times\text{Leb}=0$. Using the Fourier expansion on $\mathbb{T}^{d}$, we can express $E$ and $F$ as:
$$
E(x,\theta,u)=\sum_{m\in\mathbb{Z}^{d}}e_{m}(x,u)e^{im\theta}\quad\text{and}\quad F(x,\theta,u)=\sum_{m\in\mathbb{Z}^{d}}f_{m}(x,u)e^{im\theta},
$$
where $e_{m}(x,u)=\int_{\mathbb{T}^{d}}E(x,\theta,u)e^{-im\theta}d\theta$ and $f_{m}(x,u)=\int_{\mathbb{T}^{d}}F(x,\theta,u)e^{-im\theta}d\theta$. Let $E_{m}(x,\theta,u):=e_{m}(x,u)e^{im\theta}$ and $F_{m}(x,\theta,u):=f_{m}(x,u)e^{im\theta}$. For any $s\in\mathbb{C}$, denote by
$$e_{m,s}(x):=\int^{r(x)}_{0}e^{-su}e_{m}(x,u)du\quad\text{and}\quad f_{m,s}(x):=\int^{r(x)}_{0}e^{-su}f_{m}(x,u)du.$$
If $E, F\in F_{\lambda}^{k}(\widehat{X}_{r})$ with $k$ large, it can be evident that $e_{m,s}$ and $f_{m,s}$ belong to $F_{\lambda}(X)$ for any $m\in\mathbb{Z}^{d}$ and any $s\in\mathbb{C}$. A direct calculation yields the following estimate. Recall the constant $C_{1}>0$ in Lemma \ref{Gibbs property}.

\begin{Lemma}\label{Lemma 3.1}
Increasing $C_{1}>0$ if necessary,  for any $E,F\in F_{\lambda}^{k}(\widehat{X}_{r})$ and any $k_{2}\in\mathbb{N}^{+}$ with $k_{2}\le k$, we have,
$$||e_{m,s}||_{\lambda}\le\dfrac{C_{1}(1+|s|)||E||_{\lambda,k_{2}}}{|m|^{k_{2}}}\quad\text{and}\quad||f_{m,s}||_{\lambda}\le\dfrac{C_{1}(1+|s|)||F||_{\lambda,k_{2}}}{|m|^{k_{2}}},$$
for any $s\in\mathbb{C}$ and any $0\not=m\in\mathbb{Z}^{d}$.
\end{Lemma}

By the Fourier expansions of $E$ and $F$, we can decompose $\rho_{E,F}$ as follows.

\begin{Lemma}\label{Lemma 3.2}
For any $E, F \in F_{\lambda}^{k}(\widehat{X}_{r})$ with $k\in\mathbb{N}^{+}$ large, we have $\rho_{E,F}(t)=\sum_{m\in\mathbb{Z}^{d}}\rho_{E_{m},F_{-m}}(t)$ for any $t\ge0$.
\end{Lemma}
\begin{proof}
If $k$ is large, then the expansions $E(x,\theta,u)=\sum_{m\in\mathbb{Z}^{d}}E_{m}(x,\theta,u)=\sum_{m\in\mathbb{Z}^{d}}e_{m}(x,u)e^{im\theta}$ and $F(x,\theta,u)=\sum_{m\in\mathbb{Z}^{d}}F_{m}(x,\theta,u)=\sum_{m\in\mathbb{Z}^{d}}f_{m}(x,u)e^{im\theta}$ are absolutely convergent. In particular, we can write
\begin{equation}\label{3.1}
\begin{aligned}
\rho_{E,F}(t)=&\sum_{m\in\mathbb{Z}^{d}}\sum_{m^{\prime}\in\mathbb{Z}^{d}}\int_{X}\int_{\mathbb{T}^{d}}\int_{0}^{r(x)}E_{m}(x,\theta,u+t)F_{m^{\prime}}(x,\theta,u)dud\theta d\mu_{\varphi}(x)\\
=&\sum_{m\in\mathbb{Z}^{d}}\sum_{m^{\prime}\in\mathbb{Z}^{d}}\int_{X}\int_{0}^{r(x)}\int_{\mathbb{T}^{d}}E_{m}(x,\theta,u+t)F_{m^{\prime}}(x,\theta,u)d\theta dud\mu_{\varphi}(x).
\end{aligned}
\end{equation}
For any $(x,u)$ and any $t\ge0$, by the definitions of $E_{m}$ and $F_{m}$, we have
$$
\int_{\mathbb{T}^{d}}E_{m}(x,\theta,u+t)F_{m^{\prime}}(x,\theta,u)d\theta=0,\quad\text{if }m\not=-m^{\prime}.
$$
We then substitute the above formula into \eqref{3.1} to complete the proof.
\end{proof}

Given $s \in \mathbb{C}$, we express $s$ as $s =a+ib$. For each $m\in\mathbb{Z}^{d}$, consider the Laplace transform $\widehat{\rho}_{E_{m},F_{-m}}(s) = \int_{0}^{\infty} e^{-st}\rho_{E_{m},F_{-m}}(t)dt$. Since $\rho_{E_{m},F_{-m}}$ is bounded, $\widehat{\rho}_{E_{m},F_{-m}}$ is analytic on $\{s=a+ib: a > 0\}$. It is classical \cite{Pol85} that we can express $\widehat{\rho}_{E_{m},F_{-m}}$  as a sum of transfer operators. This can be achieved using the following result.

\begin{Lemma}\label{Lemma 3.3}
	For any $s=a+ib$ with $a>0$, we have
	$$\widehat{\rho}_{E_{m},F_{-m}}(s)=\int_{0}^{|r|_{\infty}}e^{-st}\Delta_{m}(t)dt+\sum_{n=1}^{\infty}\int_{X}e^{-sr_{n}}e^{im\Theta_{n}}e_{m,s}\circ\sigma^{n}\cdot f_{-m,-s}d\mu_{\varphi},$$
	where $\Delta_{m}(t)=\int_{X}\int_{\mathbb{T}^{d}}\int_{0}^{\max\{0,r-t\}}E_{m}\circ\phi_{t}\cdot F_{-m}dud\theta d\mu_{\varphi}$.
\end{Lemma}
\begin{proof}
	By definition, 
	\begin{equation}\label{3.2}
		\begin{aligned}
			&\widehat{\rho}_{E_{m},F_{-m}}(s)=\int_{0}^{\infty}e^{-st}\rho_{E_{m},F_{-m}}(t)dt\\
			=&\int_{0}^{\infty}e^{-st}\Delta_{m}(t)dt+\int_{0}^{\infty}\int_{X}\int_{\mathbb{T}^{d}}\int_{\max\{0,r(x)-t\}}^{r(x)}e^{-st}E_{m}(x,\theta,u+t)F_{-m}(x,\theta,u)dud\theta d\mu_{\varphi}(x)dt.
		\end{aligned}
	\end{equation}
	By the definition of $\Delta_{m}(t)$, we have $\Delta_{m}(t)=0$ when $t\ge|r|_{\infty}$, and therefore
	\begin{equation}\label{3.3}
		\int_{0}^{\infty}e^{-st}\Delta_{m}(t)dt=\int_{0}^{|r|_{\infty}}e^{-st}\Delta_{m}(t)dt.
	\end{equation}
	On the other hand, we have
	\begin{equation*}
		\begin{aligned}
			&\int_{0}^{\infty}\int_{X}\int_{\mathbb{T}^{d}}\int_{\max\{0,r(x)-t\}}^{r(x)}e^{-st}E_{m}(x,\theta,u+t)F_{-m}(x,\theta,u)dud\theta d\mu_{\varphi}(x)dt\\
			=&\int_{X}\int_{\mathbb{T}^{d}}\int_{0}^{r(x)}\int_{t\ge r(x)-u}e^{-st}E_{m}(x,\theta,u+t)F_{-m}(x,\theta,u)dtdud\theta d\mu_{\varphi}(x)\\
			=&\int_{X}\int_{\mathbb{T}^{d}}\int_{t^{\prime}\ge r(x)}e^{-st^{\prime}}E_{m}(x,\theta,t^{\prime})dt^{\prime}\int_{0}^{r(x)}e^{su}F_{-m}(x,\theta,u)dud\theta d\mu_{\varphi}(x)\\
			=&\int_{X}\int_{\mathbb{T}^{d}}\sum_{n=1}^{\infty}\bigg[\int_{r_{n}(x)}^{r_{n+1}(x)}e^{-st^{\prime}}E_{m}(x,\theta,t^{\prime})dt^{\prime}\bigg]\bigg[\int_{0}^{r(x)}e^{su}F_{-m}(x,\theta,u)du\bigg]d\theta d\mu_{\varphi}(x)\\
			=&\sum_{n=1}^{\infty}\int_{X}\int_{\mathbb{T}^{d}}e^{-sr_{n}}e^{im\Theta_{n}(x)}\bigg[\int_{0}^{r(\sigma^{n}(x))}e^{-su}e_{m}(\sigma^{n}(x),u)e^{im\theta}du\bigg]\bigg[\int_{0}^{r(x)}e^{su}f_{-m}(x,u)e^{-im\theta}du\bigg]d\theta d\mu_{\varphi}\\
			=&\sum_{n=1}^{\infty}\int_{X}e^{-sr_{n}}e^{im\Theta_{n}}e_{m,s}\circ\sigma^{n}f_{-m,-s}d\mu_{\varphi}\\
		\end{aligned}
	\end{equation*}
	Now we can substitute the above equality and \eqref{3.3} into \eqref{3.2} to complete the proof.
\end{proof}

Recalling that \(\mu_{\varphi}\) is the Gibbs measure of \(\varphi\) on \(X\) and \(\mu\) is the Gibbs measure of \(\varphi\) on \(X^{+}\), and we have \(\Pi_{+}^{*}\mu_{\varphi} = \mu\), where \(\Pi_{+}:X\to X^{+}\) is the coordinate projection.  It is a classical result that we can decompose the integral \(\int_{X} e^{-sr_{n}} e^{im\Theta_{n}} e_{m,s}\circ\sigma^{n}\cdot f_{-m,-s} \, d\mu_{\varphi}\) in Lemma \ref{Lemma 3.3} as a sum of integrals over the one-sided space \(X^{+}\) by expressing \(e_{m,s} = \sum_{p\in\mathbb{N}} e_{m,s,p}\) with \(e_{m,s,p}\in F_{\lambda}(X^{+})\) \cite{Dol98a}.\footnote{ There is a classical result which states that any function $h\in F_{\lambda}(X^{+})$ can be decomposed into $\sum_{p\in\mathbb{N}} h_{p}$ with $h_{p}$ belongs to $F_{\lambda}(X^{+})$ and enjoys some decay estimate. This step is used to handle the integration from the two-sided space to the one-sided space. The proof of this decomposition can be found in Appendix 1 in \cite{Dol98a}. For more details about using this decomposition , one can also refer to \cite[p5050]{Pol24}.} Thus, if necessary, we can make use of this decomposition and assume \(e_{m,s}\) and \(f_{-m,-s}\) belong to \(F_{\lambda}(X^{+})\). Furthermore, it is known \cite{Bow08} that \(\mu\) is the eigenmeasure of the transfer operator $\mathcal{L}_{\varphi}$ of \(\varphi\), defined by
$$\mathcal{L}_{\varphi}:F_{\lambda}(X^{+})\to F_{\lambda}(X^{+}),\quad\mathcal{L}_{\varphi}h(x) = \sum_{\sigma(y)=x} e^{\varphi(y)} h(y).$$
We can further assume \(\varphi\) is normalized, i.e., \(\mathcal{L}_{\varphi}1 = 1\). In particular, this implies \(P(\varphi) = 0\) and \(\mathcal{L}_{\varphi}^{*}\mu = \mu\), where \(P(\varphi)\) represents the pressure of \(\varphi\). Combining this with Lemma \ref{Lemma 3.3}, the above analysis yields the following.

\begin{Lemma}\label{Lemma 3.4}
	For any $s=a+ib$ with $a>0$, we have
	$$\widehat{\rho}_{E_{m},F_{-m}}(s)=\int_{0}^{|r|_{\infty}}e^{-st}\Delta_{m}(t)dt+\sum_{n=1}^{\infty}\int_{X^{+}}e_{m,s} \mathcal{L}_{\varphi-sr+im\Theta}^{n}f_{-m,-s}d\mu.$$
\end{Lemma}

Recall the constant $C_{1}>0$ in Lemma \ref{Lemma 3.1}. For each $b\in\mathbb{R}$ and $0\neq m\in\mathbb{Z}^{d}$, let \(b_{m}=|b|+C_{1}|m|\). Denote $\mathcal{L}_{b,m}:=\mathcal{L}_{\varphi-ibr+im\Theta}$. In the next section, we will utilize the condition obtained in Lemma \ref{Lemma 2.10} to prove the following Dolgopyat type estimate, which is a version of \cite[Proposition 7.3.8]{Pol24} adapted to the case of $\mathbb{T}^{d}$.

\begin{Proposition}\label{Dolgopyat type estimate}
There exist $C_{2} > 0$, $C_{3}>0$ and $C_{4} > 0$ such that for any $(b,m)$ with $m\not=0$ and any $h \in F_{\lambda}(X^{+})$, we have $||\mathcal{L}_{b,m}^{C_{2}\log b_{m}}h||_{b_{m}} \le (1 - b_{m}^{-C_{3}}) ||h||_{b_{m}}$, where $||h||_{b_{m}}=\max\{|h|_{\infty},|h|_{Lip}/C_{4}b_{m}\}$.
\end{Proposition}

By a perturbation estimate, we derive the following result concerning $\mathcal{L}_{\varphi-sr+im\Theta}$.

\begin{Corollary}\label{Cor 3.6}
There exists $C_{5} >0$ such that for any $E, F\in F_{\lambda}^{k}(X)$ and any $m\in\mathbb{Z}^{d}$ with $m\not=0$, we have $\widehat{\rho}_{E_{m},F_{-m}}$ is analytic in the region of $s=a+ib$ with $|a|\le b_{m}^{-C_{5}}$ and satisfies  $|\widehat{\rho}_{E_{m},F_{-m}}(s)|\le C_{5}b_{m}^{C_{5}}|e_{m,s}|_{\infty}||f_{-m,-s}||_{b_{m}}$.
\end{Corollary}
\begin{proof}
The proof is almost identical to that of \cite[Corollary 7.3.9]{Pol24}. By Proposition \ref{Dolgopyat type estimate}, we have
\begin{equation}\label{3.3.3}
	||\mathcal{L}_{\varphi-sr+im\Theta}^{C_{2}\log b_{m}}h||_{b_{m}}=||\mathcal{L}_{b,m}^{C_{2}\log b_{m}}(e^{-ar_{C_{2}\log b_{m}}}h)||_{b_{m}}\le (1-b_{\pi}^{-C_{3}})||e^{-ar_{C_{2}\log b_{m}}}h||_{b_{m}}.
\end{equation}
Fix a uniform constant $C_{4}^{\prime}>0$. On the one hand, when $|a|\le b_{m}^{-C_{4}^{\prime}}$, 
\begin{equation}\label{3.3.4}
	|e^{-ar_{C_{2}\log b_{m}}}h|_{\infty} \le (1+2b_{m}^{-C_{4}^{\prime}+1})|h|_{\infty}.
\end{equation}
On the other hand, also provided $|a|\le b_{m}^{-C_{4}^{\prime}}$,
\begin{equation}\label{3.3.5}
	|e^{-ar_{C_{2}\log b_{m}}}h|_{\text{Lip}}\le (1+2b_{m}^{-C_{4}^{\prime}+1})(b_{m}^{-C_{4}+1+2\log\lambda^{-1}}|h|_{\infty}+|h|_{\text{Lip}}).
\end{equation}
Substituting \eqref{3.3.4} and \eqref{3.3.5} into \eqref{3.3.3} and choosing $C_{4}^{\prime}>0$ large, we obtain
\begin{equation*}
	||\mathcal{L}_{\varphi-sr+im\Theta}^{C_{2}\log b_{m}}h||_{b_{m}}\le(1-b_{m}^{-C_{3}})(1+2b_{m}^{-C_{4}^{\prime}+1}) (1+b_{\pi}^{-C_{4}^{\prime}+1+2\log\lambda^{-1}})||h||_{b_{m}}\le (1-b_{\pi}^{-4C_{4}^{\prime}})||h||_{b_{\pi}}.
\end{equation*}
The above estimate implies that $||\mathcal{L}_{\varphi-sr+im\Theta}(1-\mathcal{L}_{\varphi-sr+im\Theta})^{-1}h||_{b_{m}}\le C_{5}b_{m}^{C_{5}}||h||_{b_{m}}$ for any $h\in F_{\lambda}(X^{+})$ and some uniform constant $C_{5}>0$. Note that $C_{5}>C_{4}^{\prime}$. Then, the result follows from Lemma~\ref{Lemma 3.4}.
\end{proof}

\begin{proof}[\textbf{Proof of Theorems \ref{Theorem 1.2} and \ref{Theorem 1.3}}]
For any \(m \in \mathbb{Z}^{d}\) with \(m \neq 0\) and any \(s = a + ib\) with \(a > 0\), by integration by parts, we have
\begin{equation}\label{3.4}
	\widehat{\rho}_{E_{m},F_{-m}}(s) = \sum_{j=1}^{k_{1}} \rho_{\partial_{u}^{j-1}E_{m},F_{-m}}(0) s^{-j} + \widehat{\rho}_{\partial_{u}^{k_{1}}E_{m},F_{-m}}(s) s^{-k_{1}},
\end{equation}
for any \(k \geq k_{1} \in \mathbb{N}^{+}\), where \(\partial_{u}\) denotes differentiation with respect to the variable \(u\). Then, as in \cite[\S 7.1]{Pol24}, modifying the correlation function \(\rho_{E_{m},F_{-m}}\) if necessary, we can assume \(\rho_{\partial_{u}^{j-1}E_{m},F_{-m}}(0) = 0\) for any \(j \leq k_{1}\). Thus, by \eqref{3.4} we have \(\widehat{\rho}_{E_{m},F_{-m}}(s) = \widehat{\rho}_{\partial_{u}^{k_{1}}E_{m},F_{-m}}(s) s^{-k_{1}}\) for any \(s = a + ib\) with \(a > 0\). Then, by Corollary \ref{Cor 3.6}, we know that \(\widehat{\rho}_{E_{m},F_{-m}}(s) = \widehat{\rho}_{\partial_{u}^{k_{1}}E_{m},F_{-m}}(s) s^{-k_{1}}\) holds as well on the curve \(\Gamma:=\{s = a + ib : a = -b_{m}^{C_{5}},\ b \in \mathbb{R}\}\). Thus, for any \(n \in \mathbb{N}^{+}\), choosing \(k_{1} = C_{5}n + 2\) and by Corollary \ref{Cor 3.6}, similarly to proof of Theorem 6 in \cite[p5043]{Pol24}, an application of the inverse Laplace transform to the curve \(\Gamma\) yields the existence of \(C_{5}^{\prime} > 0\) such that
\begin{equation}\label{3.5}
	|\rho_{E_{m},F_{-m}}(t)| \leq C_{5}^{\prime} |m|^{C_{5}} t^{-n} |\partial_{u}^{k_{1}}e_{m,s}|_{\infty} ||f_{-m,-s}||_{b_{m}}.
\end{equation}
Now, using Lemma \ref{Lemma 3.1} and following the same argument as in \cite[\S 7.4]{Pol24}, for any \(E,F \in F_{\lambda}^{k}(X^{+})\) with \(k = k_{1} + k_{2}\) and \(k_{2} = C_{5} + d + 1\), we have
\[
\sum_{0 \neq m \in \mathbb{Z}^{d}} |\rho_{E_{m},F_{-m}}(t)| \leq 2C_{5}^{\prime} t^{-n} ||E||_{\lambda,k} ||F||_{\lambda,k}.
\]
For the case where \(m = 0\), we have that \(\rho_{E_{0},F_{0}}(t)\) exhibits rapid decay, which follows from the rapid mixing property of the underlying hyperbolic flow \(g_{t}\). Thus, by Lemma \ref{Lemma 3.2}, we have proved that \(\phi_{t}\) is rapidly mixing, and therefore \(f_{t}\) is rapidly mixing as well by Lemma~\ref{Lemma 2.9}, completing the proof of Theorem \ref{Theorem 1.2}.

Indeed, following the same approach as the proof of \cite[Theorem 4]{Pol24} in \cite[\S 8]{Pol24}, the same estimates on \(\mathcal{L}_{\varphi+sr+im\Theta}\) in Proposition \ref{Dolgopyat type estimate} and Corollary \ref{Cor 3.6} ensure the superpolynomial equidistribution of \(\phi_{t}\), and thus the superpolynomial equidistribution of \(f_{t}\) as well. This completes the proof of Theorem~\ref{Theorem 1.3}.
\end{proof}

\section{Proof of Proposition \ref{Dolgopyat type estimate}}\label{sec 4}

We follow the argument in \cite[\S 7.3]{Pol24} to prove Proposition \ref{Dolgopyat type estimate}. The main ingredient in the proof is a cancellation estimate stated in Lemma \ref{Lemma 3.4.4} which will be proved in the next section. 

Recalling, we choose some suitable $C_{1}>0$ such that  \(b_{m}=|b|+C_{1}|m|_{1}\) is large sufficiently for any \((b,m)\) with \(m\neq 0\). We begin with the following Lasota-Yorke inequality.

\begin{Lemma}\label{Lemma 3.4.1}
	There exists $C_{6}>0$ such that for any $(b,m)$ with $m\not=0$, any $h\in F_{\lambda}(X^{+})$ and any $n\ge\mathbb{N}^{+}$,
	$$|\mathcal{L}_{b,m}^{n}h|_{Lip}\le C_{6}b_{m}|h|_{\infty}+\lambda^{n}|h|_{Lip}.$$
\end{Lemma}
\begin{proof}
This is the version of \cite[Lemma 7.3.1]{Pol24} in the case of $\mathbb{T}^{d}$.
\end{proof}

Recalling, for any \((b,m)\) with \(m\neq 0\), we introduce a norm on $F_{\lambda}(X^{+})$ by \(||h||_{b_{m}}=\max\{|h|_{\infty},|h|_{\text{Lip}}/C_{4}b_{m}\}\). Fix a $\lambda^{\prime}\in(\lambda,1)$, and then we choose \(C_{4}\geq 1\) such that \(C_{6}/C_{4}+\lambda\le\lambda^{\prime}\). 

\begin{Corollary}\label{Cor 3.4.2}
For any \((b,m)\) with \(m\neq 0\), any \(n\ge\mathbb{N}^{+}\), and any \(h\in F_{\lambda}(X^{+})\), we have \(\frac{|\mathcal{L}^{n}_{b,m}h|_{\text{Lip}}}{C_{4}b_{m}}\le\lambda^{\prime}||h||_{b_{m}}\) and \(||\mathcal{L}^{n}_{b,m}h||_{b_{m}}\le||h||_{b_{m}}\).
\end{Corollary}
\begin{proof}
By Lemma \ref{Lemma 3.4.1}, for any  \((b,m)\) with \(m\neq 0\), any \(n\ge\mathbb{N}^{+}\), and any \(h\in F_{\lambda}(X^{+})\),
\begin{equation*}
		\dfrac{|\mathcal{L}^{n}_{b,m}h|_{\text{Lip}}}{C_{4}b_{m}}\le \dfrac{C_{6}}{C_{4}}|h|_{\infty}+\dfrac{\lambda^{n}}{C_{4}b_{m}}|h|_{\text{Lip}}.
	\end{equation*}
	In particular,
	\begin{equation*}
		\dfrac{|\mathcal{L}^{n}_{b,m}h|_{\text{Lip}}}{C_{4}b_{m}}\le\bigg(\dfrac{C_{6}}{C_{4}}+\lambda^{n}\bigg)||h||_{b_{m}}.
	\end{equation*}
	Note that \(|\mathcal{L}^{n}_{b,m}h|_{\infty}\le|h|_{\infty}\). Thus, provided \(\frac{C_{6}}{C_{4}}+\lambda\le\lambda^{\prime}\), we have \(||\mathcal{L}^{n}_{b,m}h||_{b_{m}}\le|h|_{b_{m}}\) which completes the proof.
\end{proof}

\begin{Lemma}\label{Lemma 3.4.3}
	For any  \((b,m)\) with \(m\neq 0\), any $h\in F_{\lambda}(X^{+})$ with $|h|_{Lip}\ge 2C_{4}b_{m}|h|_{\infty}$ and any $n\ge\mathbb{N}^{+}$, we have $||\mathcal{L}^{n}_{b,m}h||_{b_{m}}\le\lambda^{\prime}||h||_{b_{m}}$.
\end{Lemma}
\begin{proof}
Since $\lambda>0$ is close to 1, we can assume $\lambda\ge 1/2$. We have $|\mathcal{L}^{n}_{b,m}h|_{\infty}\le|h|_{\infty}\le|h|_{Lip}/2C_{4}b_{m}\le2^{-1}||h||_{b_{m}}\le\lambda||h||_{b_{m}}$. Thus, by Corollary \ref{Cor 3.4.2}, we have $||\mathcal{L}^{n}_{b,m}h||_{b_{m}}\le\lambda^{\prime}||h||_{b_{m}}$.
\end{proof}

The condition obtained in Lemma \ref{Lemma 2.10} for the symbolic flow $\phi_{t}$ is enough to ensure we get the following cancellation of terms in a transfer operator.

\begin{Lemma}\label{Lemma 3.4.4}
There exist $C_{7}>0, C_{8}>0$ and $C_{9}>0$ such that for any  \((b,m)\) with \(m\neq 0\) and any $h\in F_{\lambda}(X^{+})$ with $|h|_{Lip}\le 2C_{4}b_{m}|h|_{\infty}$, there exists a subset $U\subset X^{+}$ with $\mu(U)\ge b_{m}^{-C_{7}}$ such that $|\mathcal{L}^{C_{8}\log b_{m}}_{b,m}h(x)|\le(1-b_{m}^{-C_{9}})|h|_{\infty}$ for any $x\in U$.
\end{Lemma}

The above lemma will be proved in the next section. As a corollary of Lemma \ref{Lemma 3.4.4}, we can obtain a cancellation of the oscillatory integral $\int_{X^{+}}|\mathcal{L}_{b,m}^{C_{8}\log b_{m}}h|d\mu$ as follows:
\begin{equation}\label{3.4.3}
	\begin{aligned}
		\int_{X^{+}}|\mathcal{L}_{b,m}^{C_{8}\log b_{m}}h|d\mu=&\int_{U}|\mathcal{L}_{b,m}^{C_{8}\log b_{m}}h|d\mu+\int_{X^{+}-U}|\mathcal{L}_{b,m}^{C_{8}\log b_{m}}h|d\mu\\
		\le&(1-b_{m}^{-C_{9}})|h|_{\infty}\mu(U)+|h|_{\infty}\mu(X^{+}-U)\\
		=&(1-b_{m}^{-C_{9}}\mu(U))|h|_{\infty}\le(1-b_{m}^{-C_{9}-C_{7}})|h|_{\infty}.
	\end{aligned}
\end{equation}
To strengthen the above $L^{1}$ contraction to a $|\cdot|_{\infty}$-contraction, we require the following lemma. 

\begin{Lemma}\label{Lemma 3.4.5}
	There exist $C_{10}>0$ and $\delta\in(0,1)$ such that for any $h\in F_{\lambda}(X^{+})$ and any $k\ge\mathbb{N}^{+}$,
	$$||\mathcal{L}_{\varphi}^{k}h||_{Lip}\le \int_{X^{+}}|h|d\mu+C_{10}\delta^{k}||h||_{Lip}.$$
\end{Lemma}
\begin{proof}
	This is a directly corollary of the spectral gap of $\mathcal{L}_{\varphi}$ acts on $F_{\lambda}(X^{+})$ \cite{Bal00}.
\end{proof}

\begin{Lemma}\label{Lemma 3.4.6}
	There exist $C_{11}>0$ and $ C_{12}>0$ such that for any $(b, m)$ with $m\not=0$ and any $h\in F_{\lambda}(X^{+})$ with $|h|_{Lip}\le 2C_{4}b_{m}|h|_{\infty}$,
	$$|\mathcal{L}_{b,m}^{(C_{12}+C_{8})\log b_{m}}h|_{\infty}\le\bigg(1-b_{m}^{-C_{11}}\bigg)|h|_{\infty}.$$
\end{Lemma}
\begin{proof}
	By \eqref{3.4.3}, Corollary \ref{Cor 3.4.2} and Lemma \ref{Lemma 3.4.5}, we have
	$$
	\begin{aligned}
		&|\mathcal{L}_{b,m}^{(C_{12}+C_{8})\log b_{m}}h|_{\infty}\le|\mathcal{L}_{\varphi}^{C_{12}\log b_{m}}|\mathcal{L}_{b,m}^{C_{8}\log b_{m}}h||_{\infty}\\
		\le&\int_{X^{+}}|\mathcal{L}_{b,m}^{C_{8}\log b_{m}}h|d\mu+C_{10}\delta^{C_{12}\log b_{m}}||\mathcal{L}_{b,m}^{C_{8}\log b_{m}}h||_{Lip}\\
		\le&(1-b_{m}^{-C_{9}-C_{7}})|h|_{\infty}+2C_{10}\delta^{C_{12}\log b_{m}}C_{4}b_{m}|h|_{\infty}\\
		\le&(1-b_{m}^{-C_{11}})|h|_{\infty}
	\end{aligned}
	$$
	which completes the proof.
\end{proof}

The above $|\cdot|_{\infty}$-contraction implies the following $||\cdot||_{b_{m}}$-contraction.

\begin{Corollary}\label{Cor 3.4.7}
	For any $(b, m)$ with $m\not=0$ and any $h\in F_{\lambda}(X^{+})$ with $|h|_{Lip}\le 2C_{4}b_{m}|h|_{\infty}$,
	$$||\mathcal{L}_{b,m}^{(C_{12}+C_{8})\log b_{m}}h||_{b_{m}}\le\bigg(1-b_{m}^{-C_{11}}\bigg)||h||_{b_{m}}.$$
\end{Corollary}
\begin{proof}
	This comes from Corollary \ref{Cor 3.4.2} and Lemma \ref{Lemma 3.4.6}.
\end{proof}

\begin{proof}[\textbf{Proof of Proposition \ref{Dolgopyat type estimate}}]
	This comes from Lemma \ref{Lemma 3.4.3} and Corollary \ref{Cor 3.4.7}.
\end{proof}

\section{Proof of Lemma \ref{Lemma 3.4.4}}\label{sec 5}

In this section, we prove Lemma \ref{Lemma 3.4.4}. For any $(b,m)$ with $0 \neq m \in \mathbb{Z}^{d}$, let $n_{b_{m}} = 4^{-1}C_{8}\log b_{m}$. We begin with the following lemma.

\begin{Lemma}\label{Lemma 4.3.1}
	For any $h\in F_{\lambda}(X^{+})$ with $|h|_{Lip}\le2C_{4}b_{m}|h|_{\infty}$, if there exist $k\in\{0,1,2\}$ and $x\in X^{+}$ such that $|\mathcal{L}_{b,m}^{kn_{b_{m}}}h(x)|\le(1-b_{m}^{-C_{9}+C_{8}|\varphi|_{\infty}})|h|_{\infty}$, then $|\mathcal{L}_{b,m}^{C_{8}\log b_{m}}h(\sigma^{(4-k)n_{b_{m}}}(x))| \le(1-b_{m}^{-C_{9}})|h|_{\infty}$
\end{Lemma}
\begin{proof}
	We compute
	$$
	\begin{aligned}
		|\mathcal{L}_{b,m}^{C_{8}\log b_{m}}h(\sigma^{(4-k)n_{b_{m}}}(x))|=&|\mathcal{L}_{b,m}^{(4-k)n_{b_{m}}}\mathcal{L}_{b,m}^{kn_{b_{m}}}h(\sigma^{(4-k)n_{b_{m}}}(x))|\\
		\le&\sum_{\sigma^{(4-k)n_{b_{m}}}(y)=\sigma^{(4-k)n_{b_{m}}}(x)}e^{\varphi_{(4-k)n_{b_{m}}}(y)}|\mathcal{L}_{b,m}^{kn_{b_{m}}}h(y)|\\
		\le&e^{\varphi_{(4-k)n_{b_{m}}}(x)}|\mathcal{L}_{b,m}^{kn_{b_{m}}}h(x)|+\sum_{y\not=x}e^{\varphi_{(4-k)n_{b_{m}}}(y)}|h|_{\infty}\\
		\le&e^{\varphi_{(4-k)n_{b_{m}}}(x)}|(1-b_{m}^{-C_{9}+C_{8}|\varphi|_{\infty}})|h|_{\infty}+\sum_{y\not=x}e^{\varphi_{(4-k)n_{b_{m}}}(y)}|h|_{\infty}\\
		\le&(1-b_{m}^{-C_{9}})|h|_{\infty}
	\end{aligned}
	$$
	which completes the proof.
\end{proof}

We will use Lemma \ref{Lemma 2.10} to prove that the condition in Lemma \ref{Lemma 4.3.1} is satisfied as follows.

\begin{Lemma}\label{Lemma 4.3.2}
For any $h\in F_{\lambda}(X^{+})$ with $|h|_{Lip}\le2C_{4}b_{m}|h|_{\infty}$, there exist $k\in\{0,1,2\}$ and $x\in X^{+}$ such that $|\mathcal{L}_{b,m}^{kn_{b_{m}}}h(x)|\le(1-b_{m}^{-C_{9}+C_{8}|\varphi|_{\infty}})|h|_{\infty}$.
\end{Lemma}

Thus, by the above lemma and Lemma \ref{Lemma 4.3.1}, we have a single point cancellation as shown in Lemma \ref{Lemma 4.3.1}. Then, exploiting the Gibbs property of $\mu$: Lemma \ref{Gibbs property}, it is not challenging to demonstrate that on a small neighbourhood $U$ of $\sigma^{(4-k)n_{b_{m}}}(x)$, which satisfies $\mu(U)\ge b_{m}^{-C_{7}}$, we have $|\mathcal{L}_{b,m}^{C_{8}\log b_{m}}h(x^{\prime})|\le (1-b_{m}^{-C_{9}})|h|_{\infty}$ for any $x^{\prime}\in U$. Indeed, according to Lemma~\ref{Lemma 4.3.1}, the estimate $|h|_{\text{Lip}} \le 2C_{4}b_{m} |h|_{\infty}$ implies $|\mathcal{L}_{b,m}^{C_{8}\log b_{m}}h|_{\text{Lip}} \le 2C_{4}b_{m} |h|_{\infty}$. Therefore, if $|\mathcal{L}_{b,m}^{C_{8}\log b_{m}}h(\sigma^{(4-k)n_{b_{m}}}(x))| \le (1-b_{m}^{-C_{9}})|h|_{\infty}$, then for any $y \in X^{+}$ with $d_{\lambda}(\sigma^{(4-k)n_{b_{m}}}(x),y) \le (4C_{4}b_{m}^{C_{9}+1})^{-1}$, we have
\[
\begin{aligned}
	|\mathcal{L}_{b,m}^{C_{8}\log b_{m}}h(y)|&\le |\mathcal{L}_{b,m}^{C_{8}\log b_{m}}h(y) - \mathcal{L}_{b,m}^{C_{8}\log b_{\pi}}h(\sigma^{(4-k)n_{b_{m}}}(x))| + (1-b_{m}^{-C_{9}})|h|_{\infty}\\
	&\le  \dfrac{1}{2}b_{m}^{-C_{9}}|h|_{\infty} + (1-b_{m}^{-C_{9}})|h|_{\infty} \\
	&= (1-2^{-1}b_{m}^{-C_{9}})|h|_{\infty}.
\end{aligned}
\]
Using the Gibbs property of $\mu$ (in Lemma \ref{Gibbs property}),  it is  easy to show that the set of these points $y$ has $\mu$-measure $\ge b_{m}^{-C_{7}}$ for some uniform constant $C_{7} > 0$. This completes the proof of Lemma~\ref{Lemma 3.4.4}.

In the remainder of this section, we prove Lemma \ref{Lemma 4.3.2}. We proceed by contradiction. Therefore, we assume that there exists $h\in F_{\lambda}(X^{+})$ with $|h|_{Lip}\le 2C_{4}b_{m}|h|_{\infty}$ such that for each $k\in\{0,1,2\}$,
\begin{equation}\label{4.3.1}
	|\mathcal{L}_{b,m}^{kn_{b_{m}}}h(x)|\ge (1-b_{m}^{-C_{9}+C_{8}|\varphi|_{\infty}})|h|_{\infty},\quad\text{for any } x\in X^{+}.
\end{equation}
Without loss of generality, replacing $h$ by $h/|h|_{\infty}$, we can assume $|h|_{\infty}=1$. Consequently, we can express
$$
\begin{aligned}
	h(x)&= |h(x)|e^{i\theta_{h,0}(x)},\\
	\mathcal{L}_{b,m}^{n_{b_{m}}}h(x)&= |\mathcal{L}_{b,m}^{n_{b_{m}}}h(x)|e^{i\theta_{h,1}(x)},\\
	\mathcal{L}_{b,m}^{2n_{b_{m}}}h(x)&= |\mathcal{L}_{b,m}^{2n_{b_{m}}}h(x)|e^{i\theta_{h,2}(x)}.
\end{aligned}
$$

\begin{Lemma}\label{Lemma 4.3.3}
	For any $x\in X^{+}$ and any $\sigma^{n_{b_{m}}}(y)=x$ we have
	$$|1-e^{i\theta_{h,0}(y)-i(br+m\Theta)_{n_{b_{m}}}(y)}e^{-i\theta_{h,1}(x)}|\le \dfrac{2}{|b_{m}|^{C_{9}/2-5C_{8}|\varphi|_{\infty}/8}}$$
	and
	$$|1-e^{i\theta_{h,1}(y)-i(br+m\Theta)_{n_{b_{m}}}(y)}e^{-i\theta_{h,2}(x)}|\le \dfrac{2}{|b_{m}|^{C_{9}/2-5C_{8}|\varphi|_{\infty}/8}}.$$
\end{Lemma}
\begin{proof}
	Since $\mathcal{L}_{\varphi}1=1$, by \eqref{4.3.1} we have
	\begin{equation}\label{4.3.2}
		\begin{aligned}
			&\sum_{\sigma^{n_{b_{m}}}(y)=x}e^{\varphi_{n_{b_{m}}}(y)}(1-|h(y)|e^{i\theta_{h,0}(y)-i(br+m\Theta)_{n_{b_{m}}}(y)}e^{-i\theta_{h,1}(x)})\\
			=&1-|\mathcal{L}_{b,m}^{n_{b_{m}}}h(x)|\\
			\le& \dfrac{1}{|b_{m}|^{C_{9}-C_{8}|\varphi|_{\infty}}}.
		\end{aligned}
	\end{equation}
	Again, by \eqref{4.3.1} and \eqref{4.3.2} we have
	\begin{equation*}
		\bigg|\sum_{\sigma^{n_{b_{m}}}(y)=x}e^{\varphi_{n_{b_{m}}}(y)}(1-e^{i\theta_{h,0}(y)-i(br+m\Theta)_{n_{b_{m}}}(y)}e^{-i\theta_{h,1}(x)})\bigg|\le \dfrac{2}{|b_{m}|^{C_{9}-C_{8}|\varphi|_{\infty}}}.
	\end{equation*}
	The above bound implies that
	\begin{equation}\label{4.3.3}
		\begin{aligned}
			&e^{\varphi_{n_{b_{m}}}(y)}\Re(1-e^{i\theta_{h,0}(y)-i(br+m\Theta)_{n_{b_{m}}}(y)}e^{-i\theta_{h,1}(x)})\\
			\le&\Re\sum_{\sigma^{n_{b_{m}}}(y)=x}e^{\varphi_{n_{b_{m}}}(y)}(1-e^{i\theta_{h,0}(y)-i(br+m\Theta)_{n_{b_{m}}}(y)}e^{-i\theta_{h,1}(x)})\\
			\le& \bigg|\sum_{\sigma^{n_{b_{m}}}(y)=x}e^{\varphi_{n_{b_{m}}}(y)}(1-e^{i\theta_{h,0}(y)-i(br+m\Theta)_{n_{b_{m}}}(y)}e^{-i\theta_{h,1}(x)})\bigg|\\
			\le&\dfrac{2}{|b_{m}|^{C_{9}-C_{8}|\varphi|_{\infty}}}
		\end{aligned}
	\end{equation}
	where $\Re$ means the real part of a complex number. Hence, by \eqref{4.3.3} we have
	$$
	\begin{aligned}
		&|1-e^{i\theta_{h,0}(y)-i(br+m\Theta)_{n_{b_{m}}}(y)}e^{-i\theta_{h,1}(x)}|\\
		\le& \sqrt{2\Re(1-e^{i\theta_{h,0}(y)-i(br+m\Theta)_{n_{b_{m}}}(y)}e^{-i\theta_{h,1}(x)})}\\
		\le& \bigg(\dfrac{4}{|b_{m}|^{C_{9}-5C_{8}|\varphi|_{\infty}/4}}\bigg)^{1/2}
	\end{aligned}
	$$
	which proves the first part. A similar argument gives the second part.
\end{proof}

\begin{Lemma}\label{Lemma 4.3.4}
	There exists $c\in\mathbb{R}$ such that for any $x\in X^{+}$ we have
	$$|e^{i(\theta_{h,2}(x)-\theta_{h,1}(x))}-e^{ic}|\le \dfrac{4}{b_{m}^{C_{9}/2-5C_{8}|\varphi|_{\infty}/8}}+\dfrac{4C_{4}}{b_{m}^{-4^{-1}C_{8}\log\lambda-1}}.$$
\end{Lemma}
\begin{proof}
	By Lemma \ref{Lemma 4.3.3}, for any $x\in X^{+}$ and any $\sigma^{n_{b_{m}}}(y)=x$ we have
	\begin{equation}\label{4.3.4}
		|e^{i(\theta_{h,2}(x)-\theta_{h,1}(x))}-e^{i(\theta_{h,1}(y)-\theta_{h,0}(y))}|\le \dfrac{4}{b_{m}^{C_{9}/2-5C_{8}|\varphi|_{\infty}/8}}.
	\end{equation}
	Fix a point $x_{0}\in X^{+}$. Since $\sigma:X^{+}\to X^{+}$ is topologically mixing, for any $x\in X^{+}$ there exists a point $y\in \sigma^{-n_{b_{m}}}(x)$ such that
	\begin{equation}\label{4.3.5}
		d_{\lambda}(y,z_{0})\le \lambda^{n_{b_{m}}}=\dfrac{1}{b_{m}^{-4^{-1}C_{8}\log\lambda}}.
	\end{equation}
	By Lemma \ref{Lemma 3.4.1}, it not difficult to show that the condition $|h|_{Lip}\le 2C_{4}b_{m}$ implies $|\mathcal{L}^{n_{b_{m}}}_{b,m}h|_{Lip}\le 2C_{4}b_{m}$ as well. In particular, by the definitions of $\theta_{h,0}$ and $\theta_{h,1}$, from \eqref{4.3.5}, we can obtain that
	$$|\theta_{h,1}(y)-\theta_{h,0}(y)-\theta_{h,1}(x_{0})+\theta_{h,0}(x_{0})|\le \dfrac{4C_{4}}{b_{m}^{-4^{-1}C_{8}\log\lambda-1}}.$$
	Applying the above estimate to \eqref{4.3.4}, we have
	$$|e^{i(\theta_{h,2}(x)-\theta_{h,1}(x))}-e^{i(\theta_{h,1}(x_{0})-\theta_{h,0}(x_{0}))}|\le \dfrac{4}{b_{m}^{C_{9}/2-5C_{8}|\varphi|_{\infty}/8}}+\dfrac{4C_{4}}{b_{m}^{-4^{-1}C_{8}\log\lambda-1}}.$$
	We can now set $c:=\theta_{h,1}(x_{0})-\theta_{h,0}(x_{0})$ to complete the proof.
\end{proof}

\begin{Lemma}\label{Lemma 4.3.5}
	For any $x\in X^{+}$ we have
	$$|e^{i(br+m\Theta)_{n_{b_{m}}}(x)-\theta_{h,1}\circ\sigma^{n_{b_{m}}}(x)+\theta_{h,1}(x))}-e^{ic}|\le \dfrac{6}{b_{m}^{C_{9}/2-5C_{8}|\varphi|_{\infty}/8}}+\dfrac{4C_{4}}{b_{m}^{-4^{-1}C_{8}\log\lambda-1}}.$$
\end{Lemma}
\begin{proof}
	By Lemmas \ref{Lemma 4.3.3} and \ref{Lemma 4.3.4}, we have
	$$
	\begin{aligned}
		&|e^{i((br+m\Theta)_{n_{b_{m}}}(x)-\theta_{h,1}\circ\sigma^{n_{b_{m}}}(x)+\theta_{h,1}(x))}-e^{ic}|\\
		\le&|e^{i((br+m\Theta)_{n_{b_{m}}}(x)-\theta_{h,2}\circ\sigma^{n_{b_{m}}}(x)+\theta_{h,1}(x)+c)}-e^{ic}|+\dfrac{4}{b_{m}^{C_{9}/2-5C_{8}|\varphi|_{\infty}/8}}+\dfrac{4C_{4}}{b_{m}^{-4^{-1}C_{8}\log\lambda-1}}\\
		\le&\dfrac{2}{|b_{m}|^{C_{9}/2-5C_{8}|\varphi|_{\infty}/8}}+\dfrac{4}{b_{m}^{C_{9}/2-5C_{8}|\varphi|_{\infty}/8}}+\dfrac{-4C_{4}}{b_{m}^{-4^{-1}C_{8}\log\lambda-1}}
	\end{aligned}
	$$
	which completes the proof.
\end{proof}
\bigskip

Now, Lemma \ref{Lemma 4.3.2} follows easily from the above estimate. We can choose $C_{9}>0$ such that 
$$\dfrac{1}{b_{m}^{C_{9}/2-5C_{8}|\varphi|_{\infty}/8}}\le\dfrac{1}{b_{m}^{-4^{-1}C_{8}\log\lambda-1}}.$$
In particular, by Lemma \ref{Lemma 4.3.5}, for any $x\in X^{+}$ we have
\begin{equation}\label{4.3.6}
	|e^{i(br+m\Theta)_{n_{b_{m}}}(x)-\theta_{h,1}\circ\sigma^{n_{b_{m}}}(x)+\theta_{h,1}(x))}-e^{ic}|\le\dfrac{10C_{4}}{b_{m}^{-4^{-1}C_{8}\log\lambda-1}}.
\end{equation}
Let $\{\sigma^{N}(x_{1})=x_{1}\}$, $\{\sigma^{N}(x_{2})=x_{2}\}$ and $\{\sigma^{N}(x_{3})=x_{3}\}$ be the three periodic points in Lemma~\ref{Lemma 2.10}, with lengths $\ell_{k}=r_{N}(x_{k})$ and holonomies $\theta_{k}=\Theta_{N}(x_{k})$, $k=1,2,3$. We can apply \eqref{4.3.6} to their periodic orbits to obtain that
\begin{equation*}
	d(n_{b_{m}}b\ell_{k}+n_{b_{m}}m\theta_{k}-Nc,\ 2\pi\mathbb{Z})\le\dfrac{10C_{4}}{b_{m}^{-4^{-1}C_{8}\log\lambda-1}},
\end{equation*}
for each $1\le k\le 3$. Thus, 
\begin{equation*}
	d(n_{b_{m}}b(\ell_{1}-\ell_{2})+n_{b_{m}}m(\theta_{1}-\theta_{2}),\ 2\pi\mathbb{Z})\le\dfrac{20C_{4}}{b_{m}^{-4^{-1}C_{8}\log\lambda-1}},
\end{equation*}
as well as
\begin{equation*}
	d(n_{b_{m}}b(\ell_{2}-\ell_{3})+n_{b_{m}}m(\theta_{2}-\theta_{3}),\ 2\pi\mathbb{Z})\le\dfrac{20C_{4}}{b_{m}^{-4^{-1}C_{8}\log\lambda-1}}.
\end{equation*}
Therefore, for some $q\in\mathbb{Z}$ and $p\in\mathbb{Z}$ we have
\begin{equation}\label{4.3.7}
	|n_{b_{m}}b(\ell_{1}-\ell_{2})+n_{b_{m}}m(\theta_{1}-\theta_{2})-2\pi p |\le\dfrac{20C_{4}}{b_{m}^{-4^{-1}C_{8}\log\lambda-1}}
\end{equation}
and
\begin{equation}\label{4.3.8}
	|n_{b_{m}}b(\ell_{2}-\ell_{3})+n_{b_{m}}m(\theta_{2}-\theta_{3})-2\pi q|\le\dfrac{20C_{4}}{b_{m}^{-4^{-1}C_{8}\log\lambda-1}}.
\end{equation}
Note that we should have $|q|\le C_{13}n_{b_{m}}b_{m}$ for a uniform constant $C_{13}>0$.

Recalling
$$\alpha^{\prime}=\dfrac{\ell_{1}-\ell_{2}}{\ell_{2}-\ell_{3}}\quad\text{and}\quad\beta^{\prime}=\dfrac{1}{2\pi}(\theta_{1}-\theta_{2})-\alpha^{\prime}\dfrac{1}{2\pi}(\theta_{2}-\theta_{3}).$$
By \eqref{4.3.7} and \eqref{4.3.8}, one can obtain that 
\begin{equation}\label{4.9}
	|q\alpha^{\prime}+n_{b_{m}}m\beta^{\prime}-p|\le \dfrac{(1+\alpha^{\prime})20C_{4}}{b_{m}^{-4^{-1}C_{8}\log\lambda-1}}.
\end{equation}
Meanwhile, by Lemma \ref{Lemma 2.10}, we know that $(\alpha^{\prime},\beta^{\prime})\in\mathbb{R}^{d+1}$ is an inhomogeneous Diophantine number, namely, there exist $C_{14}>0$ and $\gamma>0$ such that for any $p\in\mathbb{Z}$ and any $0\not=(q,m)\in\mathbb{Z}^{d+1}$,
\begin{equation}\label{4.10}
|q\alpha^{\prime}+m\beta^{\prime}-p|\ge C_{14}(|q|+|m|)^{-\gamma}.
\end{equation}
Note that $b_{m}=|b|+C_{1}|m|$ and $|q|\le C_{13}n_{b_{m}}b_{m}$. By choosing $C_{8}>0$ large enough, we can further bound \eqref{4.9} as follows:
\begin{equation*}
	|q\alpha^{\prime}+n_{b_{m}}m\beta^{\prime}-p|\le \dfrac{(1+\alpha^{\prime})20C_{4}}{b_{m}^{-4^{-1}C_{8}\log\lambda-1}}\le C_{14}(|q|+|n_{b_{m}}m|)^{-\gamma}.
\end{equation*}
which contradicts \eqref{4.10} and thus completing the proof of Lemma \ref{Lemma 4.3.2}.

 \end{document}